\documentclass[a4paper, 11pt, dvipsnames, svgnames]{article}

\usepackage[utf8]{inputenc}
\usepackage[top=1.0in, bottom=1.0in, left=1.0in, right=1.0in]{geometry}
\usepackage[all]{xy}
\usepackage{amsfonts}
\usepackage{amsmath}
\usepackage{amssymb}
\usepackage{amsthm}
\usepackage{calrsfs}
\usepackage{enumerate}
\usepackage{enumitem}
\usepackage{xparse, etoolbox}
\usepackage{ifthen}
\usepackage{mathtools, nccmath, textcomp}
\usepackage{relsize}
\usepackage{rsfso}
\usepackage{stmaryrd}
\usepackage{sectsty}
\usepackage{setspace}
\usepackage[nottoc]{tocbibind}
\usepackage{tensor}
\usepackage{tikz}
\usepackage{tikz-cd}
\usepackage[sf, bf]{titlesec}	
\usepackage[titles]{tocloft}
\usepackage{xcolor}
\usepackage[hyperref, backend=biber, backrefstyle=none, style=alphabetic, sorting=nyt, doi=false, url=false, maxbibnames=99, maxalphanames=99]{biblatex}
\usepackage[bbgreekl]{mathbbol}

\usepackage[T1]{fontenc}
\usepackage{lmodern}
\AtBeginBibliography{\small}

\definecolor{bondiblue}{rgb}{0.1, 0.58, 0.71}
\usepackage{hyperref}
\hypersetup{
	colorlinks=true,
	linkcolor=WildStrawberry,
	citecolor=bondiblue,
	urlcolor=black,
}

\usepackage{imakeidx}
\makeindex

\usepackage{fancyhdr}

\newcommand{\addperiod}[1]{#1.}
\titleformat{\section}[block]{\scshape\Large\filcenter}{\thesection.}{1em}{}
\titleformat{\subsection}[runin]{\normalfont\large\bfseries}{\thesubsection.}{1em}{\addperiod}
\titleformat{\subsubsection}[runin]{\normalfont\bfseries}{\thesubsubsection.}{1em}{\addperiod}

\numberwithin{equation}{section}

\DeclareSymbolFontAlphabet{\mathbbl}{bbold}

\makeatletter
\def\keywords{\xdef\@thefnmark{}\@footnotetext}
\makeatother

\theoremstyle{plain}
\newtheorem{thm}{Theorem}[section]
\newtheorem{cor}[thm]{Corollary}
\newtheorem{lem}[thm]{Lemma}
\newtheorem{prop}[thm]{Proposition}
\newtheorem*{ques}{Question}

\theoremstyle{definition}
\newtheorem{defi}[thm]{Definition}

\newtheorem{rem}[thm]{Remark}
\newtheorem*{nota}{Notation}
\AtEndEnvironment{const}{\qed}

\theoremstyle{remark}

\newenvironment{enumarabicup}
{\begin{enumerate}[font=\upshape, labelindent=\parindent, label=(\arabic*)]}
{\end{enumerate}}

\DeclareMathAlphabet{\pazocal}{OMS}{zplm}{m}{n}
\DeclareMathAlphabet{\dutchcal}{U}{dutchcal}{m}{n}

\newcommand{\isomorphic}{\xrightarrow{\hspace{0.5mm} \sim \hspace{0.5mm}}}
\newcommand{\lisomorphic}{\xleftarrow{\hspace{0.5mm} \sim \hspace{0.5mm}}}

\newcommand{\calc}{\mathcal{C}}
\newcommand{\cald}{\mathcal{D}}

\newcommand{\pazs}{\pazocal{S}}

\def\CC{\mathbb{C}}
\def\FF{\mathbb{F}}

\def\NN{\mathbb{N}}
\def\QQ{\mathbb{Q}}

\def\ZZ{\mathbb{Z}}

\newcommand{\Ainf}{A_{\inf}}

\newcommand{\Acrys}{A_{\textup{cris}}}

\newcommand{\AF}{A_F}
\newcommand{\BF}{B_F}
\newcommand{\DF}{D_F}
\newcommand{\NF}{N_F}
\newcommand{\BdR}{B_{\textup{dR}}}
\newcommand{\Bcrys}{B_{\textup{cris}}}

\newcommand{\action}{\textrm{-action}}

\newcommand{\cohomology}{\textrm{-cohomology}}

\newcommand{\cont}{\textup{cont}}

\newcommand{\crys}{\textup{cris}}

\newcommand{\Dcrys}{D_{\textup{cris}}}
\newcommand{\Vcrys}{V_{\textup{cris}}}

\newcommand{\equivariant}{\textrm{-equivariant}}
\newcommand{\etale}{\textup{\'et}}

\newcommand{\Ext}{\textup{Ext}}
\newcommand{\Fil}{\textup{Fil}}

\newcommand{\Fr}{\textup{Frac}}

\newcommand{\Gal}{\textup{Gal}}
\newcommand{\gr}{\textup{gr}}
\newcommand{\Hom}{\textup{Hom}}

\newcommand{\kert}{\textup{Ker}\hspace{0.3mm}}

\newcommand{\linear}{\textup{-linear}}

\newcommand{\lattice}{\textrm{-lattice}}
\newcommand{\module}{\textrm{-module}}
\newcommand{\modules}{\textrm{-modules}}
\newcommand{\mubar}{\overline{\mu}}

\newcommand{\MF}{\textup{MF}}

\newcommand{\Mod}{\textup{-Mod}}

\newcommand{\Finfty}{F_{\infty}}
\newcommand{\OFinfty}{O_{F_{\infty}}}

\newcommand{\OFbar}{O_{\overline{F}}}
\newcommand{\padic}{p\textrm{-adic}}

\newcommand{\ptilde}{\tilde{p}}
\newcommand{\rank}{\textup{rk}}
\newcommand{\Rep}{\textup{Rep}}

\newcommand{\representation}{\textrm{-representation}}

\newcommand{\syn}{\textup{syn}}

\newcommand{\TF}{T_F}
\newcommand{\VF}{V_F}

\newcommand{\textmod}{\textup{ mod }}

\newcommand{\torsion}{\textrm{-torsion}}

\newcommand{\wa}{^{\textup{wa}}}

\title{\vspace{-5mm}\textsc{Crystalline part of the Galois cohomology of crystalline representations}}
\author{\textsc{Abhinandan} \\ \footnotesize{IMJ-PRG, Sorbonne Universit\'e, 4 Place Jussieu, Paris, France} \\ \footnotesize{E-mail: \href{abhinandan@imj-prg.fr}{abhinandan@imj-prg.fr}}}

\newcommand{\Addresses}{{
  \footnotesize

  \rule{2cm}{0.4pt}\vspace{2mm}

  \textsc{Abhinandan}\par\nopagebreak
  \textsc{IMJ-PRG, Sorbonne Universit\'e, 4 Place Jussieu, Paris, France}\par\nopagebreak\vspace{-0.7mm}
  \textit{E-mail}: \footnotesize{\href{abhinandan@imj-prg.fr}{abhinandan@imj-prg.fr}}, \textit{Web}: \footnotesize{\href{https://abhinandan.perso.math.cnrs.fr/}{https://abhinandan.perso.math.cnrs.fr/}}
}}

\date{}

\begin{document}

\fontdimen2\font=0.3em
\pagenumbering{arabic}


\maketitle
{
	\vspace{-5mm}
	\textsc{Abstract.} For $p \geqslant 3$ and an unramified extension $F/\QQ_p$ with perfect residue field, we define a syntomic complex with coefficients in a Wach module over a certain period ring for $F$.
	We show that our complex computes the crystalline part of the Galois cohomology (in the sense of Bloch and Kato) of the associated crystalline representation of the absolute Galois group of $F$.
	Furthermore, we establish that Wach modules of Berger naturally descend over to a smaller period ring studied by Fontaine and Wach.
	This enables us to define another syntomic complex with coefficients, and we show that its cohomology also computes the crystalline part of the Galois cohomology of the associated representation.
}


\section{Introduction}

Let $F$ be an unramified extension of $\QQ_p$ with perfect residue field, and let $G_F$ denote the absolute Galois group of $F$.
One of the main goals of $\padic$ Hodge theory is to classify $\padic$ representations of $G_F$ arising from geometry, for example, \textit{crystalline}, semistable, de Rham etc.
The notion of $\padic$ crystalline representations, defined by Fontaine in \cite{fontaine-reps}, is meant to capture the idea of ``good reduction'' of algebraic varieties defined over $F$.

To understand $\padic$ representations more explicitly, Fontaine intiated different programs aiming to describe $\padic$ representations in terms of certain semilinear algebraic objects.
In \cite{fontaine-colmez} Colmez and Fontaine showed that the category of $\padic$ crystalline representations of $G_F$ is equivalent to the category of weakly admissible filtered $\varphi\modules$ over $F$ (see Section \ref{subsec:padicreps}).
On the other hand, in \cite{fontaine-phigamma}, Fontaine showed that the category of all $\padic$ representations of $G_F$ is equivalent to the category of \'etale $(\varphi, \Gamma_F)\modules$, where $\Gamma_F \isomorphic \ZZ_p^{\times}$.
Fontaine's equivalence was further refined to establish an equivalence between the category of $\padic$ crystalline representations of $G_F$ and the category of Wach modules, where a Wach module is a certain lattice inside the \'etale $(\varphi, \Gamma_F)\module$ associated to the representation (see \cite{fontaine-phigamma, wach-free, colmez-hauteur, berger-limites}).
All categorical equivalences mentioned above are exact, and it is natural to ask the following:
\begin{ques}
	Let $V$ be a $\padic$ crystalline representation of $G_F$.
	Can one (partially) compute the continuous Galois cohomology of $V$ in terms of the associated filtered $\varphi\module$, resp.\ the \'etale $(\varphi, \Gamma_F)\module$, resp.\ the Wach module?
\end{ques}

In \cite{bloch-kato}, Bloch and Kato defined a complex using the filtered $\varphi\module$ $\Dcrys(V)$ associated to $V$, and showed that their complex computes the \textit{crystalline part} of the continuous Galois cohomology of $V$ (see Section \ref{subsec:dcrys_complex_bks}).
Moreover, in \cite{herr-cohomologie}, Herr defined a complex in terms of the \'etale $(\varphi, \Gamma_F)\module$ associated to $V$ (not necessarily crystalline), and showed that his complex computes the continuous Galois cohomology of $V$ (see Section \ref{subsec:fontaine_herr_complex}).
The main objective of this article is to answer the \textit{open question} above for Wach modules, i.e.\ in this article we will define a syntomic complex in terms of the Wach module associated to $V$, and show that our complex computes the \textit{crystalline part} of the continuous Galois cohomology of $V$ (see Theorem \ref{intro_thm:syntomic_blochkatoselmer_af+} and Theorem \ref{intro_thm:syntomic_blochkatoselmer_s}).
Our results are related to the results of \cite{bloch-kato} (see Remark \ref{rem:syntomic_to_blochkato_af+} and Remark \ref{rem:syntomic_to_blochkato_s}) and \cite{herr-cohomologie} (see Remark \ref{rem:syntomic_to_fontaineherr}), however, our constructions and proofs are direct and independent of \cite{bloch-kato, herr-cohomologie}.

At this point, let us also remark that the question of computing Bloch--Kato Selmer groups have previously been taken up in prior works using different objects in $\padic$ Hodge theory.
For example, in \cite{iovita-marmora}, Iovita and Marmora studied similar questions using the theory of Breuil modules, and in \cite{pottharst}, Pottharst studied a subspace of the Bloch--Kato Selmer group using the theory of $(\varphi, \Gamma)\textrm{-modules}$ over the Robba ring.
To describe our results in more detail, we begin by recalling the definition given by Bloch and Kato.

\subsection{Bloch--Kato Selmer groups}\label{subsec:dcrys_complex_bks}

Let $V$ be a $\padic$ crystalline representation of $G_F$.
In \cite[Section 3, Equation (3.7.2)]{bloch-kato}, the authors defined Bloch--Kato Selmer groups of $V$ as a subspace inside the continuous $G_F\cohomology$ of $V$, i.e.\ $H^k_f(G_F, V) \subset H^k(G_F, V)$, for $k \in \NN$.
Morally, the Bloch--Kato Selmer group picks out the \textit{crystalline part} of the Galois cohomology of $V$ (see Remark \ref{rem:h1f_extclass}).
More precisely, let $V$ be a $\padic$ crystalline representation of $G_F$.
Then, from \cite[Corollary 3.8.4]{bloch-kato} we have that the complex
\begin{equation}\label{intro_eq:dcrys_complex_bks}
	\cald^{\bullet}(\Dcrys(V)) \colon \Fil^0 \Dcrys(V) \xrightarrow{\hspace{1mm} 1-\varphi \hspace{1mm}} \Dcrys(V),
\end{equation}
computes the crystalline part of the Galois cohomology of $V$, i.e.\ we have natural isomorphisms $H^k(\cald^{\bullet}(\Dcrys(V))) \isomorphic H^k_f(G_F, V)$ for each $k \in \NN$ (see Corollary \ref{cor:dcrys_complex_bks} for a proof).

Let us note that, in \cite{bloch-kato}, Bloch and Kato defined the Selmer group, and used it as an important ingredient for the computation of Tamagawa numbers.
Moreover, in \cite{benois-berger}, Benois and Berger proved some conjectures in the Iwasawa theory of crystalline representations, formulated in terms of the local Tamagawa numbers, using techniques employed in the theory of $(\varphi, \Gamma)\textrm{-modules}$.
Then, it is reasonable to expect that when $V$ is a crystalline representation with arbitrary Hodge--Tate weights, then the syntomic complex of Theorem \ref{intro_thm:syntomic_blochkatoselmer_af+} below and its integral variant (see Remarks \ref{rem:bhatt_lurie_syntomic} and \ref{rem:integral_bks}) would provide one with new tools to study the local Tamagawa number in this case.

\subsection{Syntomic complexes and Galois cohomology}\label{subsec:syncompl_galcoh}

Let $p \geqslant 3$ and set $\Finfty := \cup_n F(\mu_{p^n})$ with $\Gamma_F := \Gal(F_{\infty}/F) \isomorphic \ZZ_p^{\times}$, via the $\padic$ cycolotomic character $\chi$.
Note that $\Gamma_F$ fits into the following exact sequence:
\begin{equation}\label{intro_eq:gammaf_es}
	1 \longrightarrow \Gamma_0 \longrightarrow \Gamma_F \longrightarrow \Gamma_{\textrm{tor}} \longrightarrow 1,
\end{equation}
where $\Gamma_0 \isomorphic 1 + p\ZZ_p$ and $\Gamma_{\textrm{tor}} \isomorphic \FF_p^{\times}$, and the projection map in \eqref{intro_eq:gammaf_es} admits a section $\Gamma_{\textrm{tor}} \isomorphic \FF_p^{\times} \rightarrow \ZZ_p^{\times} \lisomorphic \Gamma_F$, where the second map is given as $a \mapsto [a]$, the Teichm\"uller lift of $a$.
We also fix a topological generator $\gamma$ of $\Gamma_0$.

Let $q$ be an indeterminate, and set $\AF^+ := O_F\llbracket q-1 \rrbracket$ equipped with a Frobenius endomorphism $\varphi$ extending the Frobenius on $O_F$ and by setting $\varphi(q) = q^p$, and an $O_F\linear$ and continuous action of $\Gamma_F$ defined by setting $g(q) = q^{\chi(g)}$, for any $g$ in $\Gamma_F$ (see Section \ref{subsec:period_rings}).
Set $\mu := q-1$, and fix the following elements inside $\AF^+$:
\begin{equation*}
	[p]_q := \tfrac{q^p-1}{q-1}, \hspace{3mm} \mu_0 := -p + \textstyle\sum_{a \in \FF_p} q^{[a]}, \hspace{3mm} \ptilde := \mu_0 + p.
\end{equation*}
Now, we set $S := O_F\llbracket \mu_0 \rrbracket = (\AF^+)^{\FF_p^{\times}} \subset \AF^+$, which is stable under the action of $\varphi$ and $\Gamma_0$ (see \cite[Section 3.1.2]{abhinandan-prismatic-wach}); we equip $S$ with the induced $(\varphi, \Gamma_0)\action$.
Our goal is to define a syntomic complex with coefficients in a Wach module over $\AF^+$ (resp.\ $S$), and explore its relationship with the Galois cohomology of the associated crystalline representation.

\subsubsection{Syntomic complex over \texorpdfstring{$\AF^+$}{-}}

Let $\Rep_{\ZZ_p}^{\crys}(G_F)$ denote the category of $\ZZ_p\textrm{-lattices}$ inside $\padic$ crystalline representations of $G_F$, and let $(\varphi, \Gamma)\Mod_{\AF^+}^{[p]_q}$ denote the category of Wach modules over $\AF^+$ (see Definition \ref{defi:wachmod_af+}).
Then, by \cite{fontaine-phigamma, wach-free, colmez-hauteur, berger-limites} we have an equivalence of categories $\Rep_{\ZZ_p}^{\crys}(G_F) \isomorphic (\varphi, \Gamma_F)\Mod_{\AF^+}^{[p]_q}$, sending $T \mapsto N_F(T)$ (see Theorem \ref{thm:crystalline_wach_af+_equivalence}).
Moreover, after inverting $p$, i.e.\ upon passing to associated isogeny categories, the Wach module functor induces an exact equivalence of categories (see Remark \ref{rem:crystalline_wach_equivalence_rational}).

Let $T$ be a $\ZZ_p\textrm{-representation}$ of $G_F$ such that $V := T[1/p]$ is crystalline.
Let $N := \NF(T)$ denote the Wach module over $\AF^+$ associated to $T$ by Theorem \ref{thm:crystalline_wach_af+_equivalence}.
Define a decreasing filtration on $N$ called the \textit{Nygaard filtration} as $\Fil^k N := \{ x \in N \textrm{ such that } \varphi(x) \in [p]_q^k N\}$, for $k \in \ZZ$.
Define an operator on $N$ as $\nabla_q := \tfrac{\gamma-1}{q-1} \colon N \rightarrow N$.
Then, for each $k \in \ZZ$ we have that $\nabla_q(\Fil^k N) \subset \Fil^{k-1} N$ (see Remark \ref{rem:gamma_minus1_image}).
\begin{defi}\label{intro_defi:syntomic_complex_af+}
	Define the \textit{syntomic complex} with coefficients in $N$ as
	\begin{equation*}
		\pazs^{\bullet}(N) \colon \Fil^0 N \xrightarrow{\hspace{1mm}(\nabla_q, 1-\varphi)\hspace{1mm}} \Fil^{-1} N \oplus N \xrightarrow{\hspace{1mm}(1-[p]_q\varphi, \nabla_q)^{\intercal}\hspace{1mm}} N,
	\end{equation*}
	where the first map is $x \mapsto (\nabla_q(x), (1-\varphi)x)$, and the second map is $(x, y) \mapsto (1-[p]_q\varphi)x - \nabla_q(y)$.
\end{defi}

Our first main result is as follows:
\begin{thm}[{Theorem \ref{thm:syntomic_blochkatoselmer_af+}}]\label{intro_thm:syntomic_blochkatoselmer_af+}
	For each $k \in \NN$, we have a natural isomorphism
	\begin{equation*}
		H^k(\pazs^{\bullet}(N))[1/p] \isomorphic H^k_f(G_F, V).
	\end{equation*}
\end{thm}
The proof of Theorem \ref{intro_thm:syntomic_blochkatoselmer_af+} is subtle.
For computing $H^0(\pazs^{\bullet}(N))$, we first show that Wach modules over $\AF^+$ canonically descend to Wach modules over $S$ (see Theorem \ref{intro_thm:wachmod_s_af+_equiv}), and then we study the action of $\Gamma_0$ on $N$ (see Lemma \ref{lem:h0_syn_blochkatoselmer}).
To prove the claim for $H^1$, we show that our complex computes the extension classes of $\AF^+$ by $N$ in the category of Wach modules over $\AF^+$ (see Proposition \ref{prop:h1syn_ext1wach_af+}), and after inverting $p$, each class gives rise to a crystalline extension class of $\QQ_p$ by $V$, and vice versa (see Corollary \ref{cor:h1syn_ext1wach_bf+}).
Finally, by explicitly studying the action on $N$, of the generator $\gamma$ of $\Gamma_0$, we show that $H^2(\pazs^{\bullet}(N))[1/p]$ vanishes (see Proposition \ref{prop:h2_syn_blochkatoselmer}).

\begin{rem}
	Note that $(N/\mu N)[1/p]$ is a $\varphi\module$ over $F$ since $[p]_q = p \textmod \mu \AF^+$, and $N/\mu N$ is equipped with a filtration $\Fil^k(N/\mu N)$ given as the image of $\Fil^k N$ under the surjection $N \twoheadrightarrow N/\mu N$.
	We equip $(N/\mu N)[1/p]$ with the induced filtration $\Fil^k((N/\mu N)[1/p]) := \Fil^k(N/\mu N)[1/p]$, and note that it is a filtered $\varphi\module$ over $F$.
	Then, from \cite[Th\'eor\`eme III.4.4]{berger-limites} and \cite[Theorem 1.7 \& Remark 1.8]{abhinandan-imperfect-wach} we have that $(N/\mu N)[1/p] \isomorphic \Dcrys(V)$ as filtered $\varphi\modules$ over $F$ (see Theorem \ref{thm:qdeformation_dcrys}).
	The preceding comparison enables us to define a morphism of complexes $\pazs^{\bullet}(N)[1/p] \rightarrow \cald^{\bullet}(\Dcrys(V))$, which induces a quasi-isomorphism (see Remark \ref{rem:syntomic_to_blochkato_af+}).
	In particular, the complex $\pazs^{\bullet}(N)$ may be regarded as a ``lifting to $\AF^+$'' of the complex $\cald^{\bullet}(\Dcrys(V))$.
\end{rem}

\begin{rem}\label{rem:syntomic_to_fontaineherr}
	Definiton \ref{intro_defi:syntomic_complex_af+} can be modified (up to isomorphism) to obtain a subcomplex of the Fontaine--Herr complex from \cite{herr-cohomologie} (see Remark \ref{rem:syntomic_to_blochkato_af+} and Remark \ref{rem:syntomic_to_blochkato_s}).
	Note that the Fontaine--Herr complex computes the Galois cohomology of a representation, while the complex in Definition \ref{intro_defi:syntomic_complex_af+} or Remark \ref{rem:syntomic_to_blochkato_af+} is concerned with capturing the crystalline part of the Galois cohomology.
	Complexes similar to the modified complex in Remark \ref{rem:syntomic_to_blochkato_af+} were studied in \cite{abhinandan-syntomic} and named syntomic complexes.
	Hence, we refer to the complex in Definition \ref{intro_defi:syntomic_complex_af+} as the syntomic complex with coefficients in $N$.
\end{rem}

\begin{rem}\label{rem:bhatt_lurie_syntomic}
	In \cite[Chapter 6]{bhatt-lurie-fgauges}, Bhatt and Lurie have defined syntomic cohomology of prismatic $F\textrm{-gauges}$ on the stack $\ZZ_p^{\syn}$, and in case of reflexive $F\textrm{-gauges}$, compared it to the Bloch--Kato Selmer groups of the associated crystalline representation of $\Gal(\overline{\QQ}_p/\QQ_p)$ (see \cite[Proposition 6.7.3]{bhatt-lurie-fgauges}).
	In light of Theorem \ref{intro_thm:syntomic_blochkatoselmer_af+} and given the prismatic interpretation of Wach modules (see \cite{abhinandan-prismatic-wach}), a natural and interesting question is to ask for a direct (integral) relationship between the syntomic complex of Definition \ref{intro_defi:syntomic_complex_af+} and the syntomic cohomology with coefficients of \cite{bhatt-lurie-fgauges}.
	The aforementioned question and generalisation of the theory above to the relative case, i.e.\ Definition \ref{intro_defi:syntomic_complex_af+} and its relationship with Galois cohomology, will be investigated in a future work.
\end{rem}

\subsubsection{Syntomic complex over \texorpdfstring{$S$}{-}}

Let $(\varphi, \Gamma_0)\Mod_{S}^{\ptilde}$ denote the category of Wach modules over $S$ (see Definition \ref{defi:wachmod_s}).
Our second main result establishes the following descent statement for Wach modules:
\begin{thm}[{Theorem \ref{thm:wachmod_s_af+_equiv}}]\label{intro_thm:wachmod_s_af+_equiv}
	The following natural functor induces an equivalence of $\otimes\textrm{-categories}$,
	\begin{align}\label{intro_eq:wachmod_s_af+_equiv}
		\begin{split}
			(\varphi, \Gamma_0)\Mod_{S}^{\ptilde} &\isomorphic (\varphi, \Gamma_F)\Mod_{\AF^+}^{[p]_q}\\
			M &\longmapsto \AF^+ \otimes_S M,
		\end{split}
	\end{align}
	with a $\otimes\textrm{-compatible}$ quasi-inverse functor given as $N \mapsto N^{\FF_p^{\times}}$.
	Moreover, the functor in \eqref{intro_eq:wachmod_s_af+_equiv} and its quasi-inverse preserve short exact sequences.
\end{thm}

By combining Theorem \ref{intro_thm:wachmod_s_af+_equiv} and Theorem \ref{thm:crystalline_wach_s_equivalence}, we obtain a natural equivalence of categories $\Rep_{\ZZ_p}^{\crys}(G_F) \isomorphic (\varphi, \Gamma_0)\Mod_S^{\ptilde}$, by sending $T \mapsto M_F(T) := \NF(T)^{\FF_p^{\times}}$ (see Theorem \ref{thm:crystalline_wach_s_equivalence}).

Let $T$ be a $\ZZ_p\textrm{-representation}$ of $G_F$ such that $V := T[1/p]$ is crystalline.
Let $M := M_F(T)$ denote the Wach module over $S$ associated to $T$ by Theorem \ref{thm:crystalline_wach_s_equivalence}.
Define a decreasing filtration on $M$ called the \textit{Nygaard filtration} as $\Fil^k M := \{ x \in M \textrm{ such that } \varphi(x) \in \ptilde^k M\}$, for $k \in \ZZ$.
Define an operator on $M$ as $\nabla_0 := \tfrac{\gamma-1}{\mu_0} \colon M \rightarrow M$.
Then, for each $k \in \ZZ$ we have $\nabla_0(\Fil^k M) \subset \Fil^{k-p+1} M$ (see Remark \ref{rem:gamma0_minus1_image}).
\begin{defi}\label{intro_defi:syntomic_complex_s}
	Define the \textit{syntomic complex} with coefficients in $M$ as
	\begin{equation*}
		\pazs^{\bullet}(M) \colon \Fil^0 M \xrightarrow{\hspace{1mm}(\nabla_0, 1-\varphi)\hspace{1mm}} \Fil^{-p+1} M \oplus M \xrightarrow{\hspace{1mm}(1-\ptilde^{p-1}\varphi, \nabla_0)^{\intercal}\hspace{1mm}} M,
	\end{equation*}
	where the first map is given as $x \mapsto (\nabla_0(x), (1-\varphi)x)$, and the second map is given as $(x, y) \mapsto (1-\ptilde^{p-1}\varphi)x - \nabla_0(y)$.
\end{defi}

Our third main result is as follows:
\begin{thm}[{Theorem \ref{thm:syntomic_blochkatoselmer_s}}]\label{intro_thm:syntomic_blochkatoselmer_s}
	For each $k \in \NN$, we have a natural isomorphism
	\begin{equation*}
		H^k(\pazs^{\bullet}(M))[1/p] \isomorphic H^k_f(G_F, V).
	\end{equation*}
\end{thm}
To prove Theorem \ref{intro_thm:syntomic_blochkatoselmer_s}, we first define a morphism of complexes $\pazs^{\bullet}(M) \rightarrow \pazs^{\bullet}(N)$, where $N = \AF^+ \otimes_S M$ is a Wach module over $\AF^+$ (see Theorem \ref{intro_thm:wachmod_s_af+_equiv}).
Then, using the $\FF_p^{\times}\textrm{-decomposition}$ of $N$ (see Remark \ref{rem:fpx_decomp}), we show that the natural map on cohomology $H^k(\pazs^{\bullet}(M)) \rightarrow H^k(\pazs^{\bullet}(N))$ is bijective for $k = 0, 1$ and injective for $k = 2$ (see Proposition \ref{intro_thm:wachmod_s_af+_equiv}).
Combining this with Theorem \ref{intro_thm:syntomic_blochkatoselmer_af+}, yields the claim.

\begin{rem}
	Note that $(M/\mu_0 M)[1/p]$ is a $\varphi\module$ over $F$ since $\ptilde = p \textmod \mu_0 S$, and $M/\mu_0 M$ is equipped with a filtration $\Fil^k(M/\mu_0 M)$ given as the image of $\Fil^k M$ under the surjection $M \twoheadrightarrow M/\mu_0 M$.
	We equip $(M/\mu_0 M)[1/p]$ with the induced filtration $\Fil^k((M/\mu_0 M)[1/p]) := \Fil^k(M/\mu_0 M)[1/p]$, and note that it is a filtered $\varphi\module$ over $F$.
	Then, in Theorem \ref{thm:mu0_deformation_dcrys} we show that $(M/\mu_0 M)[1/p] \isomorphic \Dcrys(V)$ as filtered $\varphi\modules$ over $F$.
	The preceding comparison enables us to define a morphism of complexes $\pazs^{\bullet}(M)[1/p] \rightarrow \cald^{\bullet}(\Dcrys(V))$, which induces a quasi-isomorphism (see Remark \ref{rem:syntomic_to_blochkato_s}).
	In particular, the complex $\pazs^{\bullet}(M)$ may be regarded as a ``lifting to $S$'' of the complex $\cald^{\bullet}(\Dcrys(V))$.
\end{rem}

\begin{rem}
	Recall that we fixed $p \geqslant 3$ at the beginning of Section \ref{subsec:syncompl_galcoh}.
	This is due to the fact that the methods developed in this paper only apply to the case $p \geqslant 3$.
	Indeed, for $p \geqslant 3$, note that in \eqref{intro_eq:gammaf_es} we have that $\Gamma_{\textrm{tor}} \isomorphic \FF_p^{\times}$, whereas for $p = 2$, we have that $\Gamma_{\textrm{tor}} \isomorphic \{\pm 1\}$ as groups.
	Therefore, our main technical result, the descent statement of Theorem \ref{intro_thm:wachmod_s_af+_equiv}, does not apply to the case $p=2$.
	Moreover, we do not know whether a claim analogous to Theorem \ref{intro_thm:wachmod_s_af+_equiv} holds for $p=2$ (even after replacing $\FF_p^{\times}$ with the group $\{\pm 1\}$).
	Furthermore, let us note that Theorem \ref{intro_thm:wachmod_s_af+_equiv} is obviously crucial for obtaining Theorem \ref{intro_thm:syntomic_blochkatoselmer_s}.
	Additionally, Theorem \ref{intro_thm:wachmod_s_af+_equiv} has also been used in the proof of Theorem \ref{intro_thm:syntomic_blochkatoselmer_af+} (see Lemma \ref{lem:h0_syn_blochkatoselmer}).
	Therefore, for $p=2$, to obtain a variation of Theorem \ref{intro_thm:syntomic_blochkatoselmer_af+}, we need a variation of Theorem \ref{intro_thm:wachmod_s_af+_equiv}, which remains elusive.
\end{rem}

\subsection{Outline of the paper}

This article consists of three main sections.
In Section \ref{sec:padicreps}, we quickly recall the necessary definitions and results on period rings, $\padic$ representations and Galois cohomology.
We begin by fixing some notations and recalling several period rings in Section \ref{subsec:period_rings}, and the theory of $\padic$ crystalline representations and $(\varphi, \Gamma)\modules$ in Section \ref{subsec:padicreps}.
In Section \ref{subsec:fontaine_herr_complex}, we recall the Fontaine--Herr complex and its relationship with the Galois cohomology of $\padic$ representations; in Section \ref{subsec:blochkatoselmer_groups} we recall the definition of Bloch--Kato Selmer groups and the construction of the Bloch--Kato complex from \eqref{intro_eq:dcrys_complex_bks}.
Section \ref{sec:wachmods} is devoted to the study of Wach modules and the Nygaard filtration on Wach modules.
In Section \ref{eq:wachmod_descent}, we first recall the definition of Wach modules over $\AF^+$ and $S$, and then prove the descent claim of Theorem \ref{intro_thm:wachmod_s_af+_equiv}.
In Section \ref{subsec:wachmods_to_crysreps}, we recall the relationship between Wach modules and crystalline representations, and in Section \ref{subsec:nyfil} we define the Nygaard filtration on Wach modules and prove several properties of the filtration that are to be used later.
In Section \ref{sec:syntomic_complex_galcoh}, we arrive at the main results on the computation of Galois cohomology.
We begin Section \ref{subsec:syntomic_complex_af+} by defining the syntomic complex over $\AF^+$ (see Definition \ref{intro_defi:syntomic_complex_af+}), and then state and prove Theorem \ref{intro_thm:syntomic_blochkatoselmer_af+}.
The proof for each cohomological degree $k = 0, 1, 2$ is carried out separately (see Lemma \ref{lem:h0_syn_blochkatoselmer}, Proposition \ref{prop:h1syn_ext1wach_af+} and Proposition \ref{prop:h2_syn_blochkatoselmer}).
Finally, in Section \ref{subsec:syntomic_complex_s} we define the syntomic complex over $S$ (see Definition \ref{intro_defi:syntomic_complex_s}) and prove Theorem \ref{intro_thm:syntomic_blochkatoselmer_s}.

\vspace{2mm}

\noindent \textbf{Acknowledgements.} 
The work presented here was partially carried out during my PhD at Universit\'e de Bordeaux.
I would like to sincerely thank my advisor Denis Benois for several discussions around the content of this article, and Takeshi Tsuji for helpful discussions around Theorem \ref{thm:wachmod_s_af+_equiv}.
I would also like to thank Luming Zhao for helpful remarks on a previous version of the article.
Finally, I would like to thank the referee for carefully reading the article and pointing out issues in the previous version, which helped in rectification of errors and improvement in the writing.
This research is partially supported by JSPS KAKENHI grant numbers 22F22711 and 22KF0094.

\section{\texorpdfstring{$\padic$}{-} representations}\label{sec:padicreps}

Let $p \geqslant 3$ be a fixed prime number, and let $\kappa$ denote a perfect field of characteristic $p$; set $O_F := W(\kappa)$ to be the ring of $p\textrm{-typical}$ Witt vectors with coefficients in $\kappa$, and $F := \Fr(O_F)$.
Let $\overline{F}$ denote a fixed algebraic closure of $F$, let $\CC_p := \widehat{\overline{F}}$ denote its $\padic$ completion, and $G_F := \Gal(\overline{F}/F)$ the absolute Galois group of $F$.
Moreover, let $\Finfty := \cup_n F(\mu_{p^n})$, and set $\Gamma_F := \Gal(F_{\infty}/F) \isomorphic \ZZ_p^{\times}$ and $H_F := \Gal(\overline{F}/\Finfty)$.
Note that the isomorphism $\chi \colon \Gamma_F \isomorphic \ZZ_p^{\times}$ is given via the $\padic$ cyclotomic character, and therefore, $\Gamma_F$ fits into the following exact sequence:
\begin{equation}\label{eq:gammaf_es}
	1 \longrightarrow \Gamma_0 \longrightarrow \Gamma_F \longrightarrow \Gamma_{\textrm{tor}} \longrightarrow 1,
\end{equation}
where we have that $\Gamma_0 \isomorphic 1 + p\ZZ_p$ and $\Gamma_{\textrm{tor}} \isomorphic \FF_p^{\times}$, and the projection map in \eqref{eq:gammaf_es} admits a section $\Gamma_{\textrm{tor}} \isomorphic \FF_p^{\times} \rightarrow \ZZ_p^{\times} \lisomorphic \Gamma_F$, where the second map is given as $a \mapsto [a]$, the Teichm\"uller lift of $a$.
We fix a topological generator $\gamma$ of $\Gamma_0$.

\begin{rem}\label{rem:fpx_decomp}
	Let $N$ be a linearly topologised compact $\ZZ_p\module$ admitting a continuous action of $\Gamma_F$.
	Then, from \cite[Section 3]{iwasawa}, the module $N$ admits an $\FF_p^{\times}\textrm{-decomposition}$ $N = \bigoplus_{i=0}^{p-1} N_i$, where $N_0 = N^{\FF_p^{\times}}$.
	Moreover, since $\Gamma_F$ is commutative, each $N_i$ is equipped with an induced continuous action of $\Gamma_0$.
\end{rem}

\subsection{Period rings}\label{subsec:period_rings}

In this section, we will quickly recall all the period rings used in this paper (see \cite{fontaine-phigamma, fontaine-corps-periodes} for details).
Let $\OFbar$ (resp.\ $\OFinfty$) denote the ring of integers of $\overline{F}$ (resp.\ $F_{\infty}$), and let $\OFbar^{\flat} := \lim_{x \mapsto x^p} \OFbar/p$ (resp.\ $\OFinfty^{\flat} := \lim_{x \mapsto x^p} \OFinfty/p$) denote its tilt (see \cite{fontaine-corps-periodes}).
Let us set $\Ainf(\OFinfty) := W(\OFinfty^{\flat})$ and $\Ainf(\OFbar) := W(\OFbar^{\flat})$ admitting the natural Frobenius on Witt vectors, and continuous (for the weak topology) actions of $\Gamma_F$ and $G_F$, respectively.
We fix $\mubar := \varepsilon-1$, where $\varepsilon := (1, \zeta_p, \zeta_{p^2}, \ldots)$ is in $\OFinfty^{\flat}$, and let $q = [\varepsilon]$, $\mu := q - 1$ and $\xi := \mu/\varphi^{-1}(\mu)$ in $\Ainf(\OFinfty)$.
Then, for any $g$ in $G_F$, we have that $g(1+\mu) = (1+\mu)^{\chi(g)}$, where $\chi$ is the $\padic$ cyclotomic character.
Moreover, note that we have a $G_F\equivariant$ surjective homomorphism of rings $\theta \colon \Ainf(\OFbar) \rightarrow \CC_p$ with $\kert \theta = \xi \Ainf(\OFbar)$.
The map $\theta$ restricts to a $\Gamma_F\equivariant$ surjective homomorphism of rings $\theta \colon \Ainf(\OFinfty) \rightarrow O_{\widehat{F}_{\infty}}$ with $\kert \theta = \xi \Ainf(\OFinfty)$.

We set $\Acrys(\OFinfty) := \Ainf(\OFinfty)\langle \xi^k/k!, k \in \NN \rangle$, and note that $t := \log(1+\mu) = \sum_{k \in \NN} (-1)^k \tfrac{\mu^{k+1}}{k+1}$ converges in $\Acrys(\OFinfty)$.
The ring $\Acrys(\OFinfty)$ is $p\torsion$ free and $t\torsion$ free, and we set $\Bcrys^+(\OFinfty) := \Acrys(\OFinfty)[1/p]$ and $\Bcrys(\OFinfty) := \Bcrys^+(\OFinfty)[1/t]$.
These rings are equipped with a Frobenius endomorphism $\varphi$, a continuous $\Gamma_F\action$ and a decreasing filtration; the ring $\Acrys^+(\OFinfty)$ and $\Bcrys^+(\OFinfty)$ are further equipped with an appropriate extension of the map $\theta$.
Next, let us set $\BdR^+(\OFinfty) := \lim_n (\Ainf(\OFinfty)[1/p])/(\kert \theta)^n$ and $\BdR(\OFinfty) := \BdR^+(\OFinfty)[1/t]$.
These rings are equipped with a $\Gamma_F\action$ and a decreasing filtration; the ring $\BdR^+(\OFinfty)$ is further equipped with an appropriate extension of the map $\theta$.
Moreover, we have $(\varphi, \Gamma_F)\equivariant$ and filtration compatible injective homomorphism of $F\textrm{-algebras}$ $\Bcrys^+(\OFinfty) \rightarrow \BdR^+(\OFinfty)$ and $\Bcrys(\OFinfty) \rightarrow \BdR(\OFinfty)$.
One can define variations of all the aforementioned rings and their properties over $\OFbar$ as well.
Furthermore, from \cite[Th\'eor\`eme 5.3.7]{fontaine-corps-periodes} recall that we have the following $(\varphi, G_F)\equivariant$ \textit{fundamental exact sequence}:
\begin{equation}\label{eq:fes_bcris}
	0 \longrightarrow \QQ_p \longrightarrow \Fil^0 \Bcrys(\OFbar) \xrightarrow{\hspace{1mm} 1-\varphi \hspace{1mm}} \Bcrys(\OFbar) \longrightarrow 0.
\end{equation}

Let us set $\AF^+ := O_F\llbracket \mu \rrbracket$, and we equip it with a Frobenius endomorphism $\varphi$, given by the Witt vector Frobenius on $O_F$ and by setting $\varphi(\mu) = (1+\mu)^p-1$, and an $O_F\linear$ and continuous action of $\Gamma_F$ given as $g(\mu) = (1+\mu)^{\chi(g)}-1$, for any $g$ in $\Gamma_F$.
Note that we have a natural injective homomorphism of rings $\AF^+ \hookrightarrow \Ainf(\OFinfty)$ compatible with Frobenius and $\Gamma_F\action$.
Set $\AF := \AF^+[1/\mu]^{\wedge}$, where ${}^{\wedge}$ denotes the $\padic$ completion.
The Frobenius endomorphism $\varphi$ and the continuous action of $\Gamma_F$ on $\AF^+$ naturally extend to respective actions on $\AF$.
Let $W\big(\CC_p^{\flat}\big)$ denote the ring of $p\textrm{-typical}$ Witt vectors with coefficients in $\CC_p^{\flat} := \Fr(\OFbar^{\flat})$, and note that $W(\CC_p^{\flat})$ is equipped with the natural Frobenius on Witt vectors and a continuous (for the weak topology) action of $G_F$.
The natural $(\varphi, \Gamma_F)\equivariant$ injective homomorphism $\AF^+ \hookrightarrow \Ainf(\OFinfty)$ extends to a natural $(\varphi, \Gamma_F)\equivariant$ injective homomorphism of rings $\AF \hookrightarrow W(\CC_p^{\flat})^{H_F} = W(F_{\infty}^{\flat})$.

Next, we recall some definitions and observations from \cite[Section 3.1]{abhinandan-prismatic-wach} by fixing the following elements inside $\AF^+$:
\begin{equation}\label{eq:pq_mu0_ptilde_defi}
	\begin{aligned}
		[p]_q &:= \tfrac{q^p-1}{q-1} = \tfrac{\varphi(\mu)}{\mu},\\
		\mu_0 &:= \textstyle\sum_{a \in \FF_p^{\times}} ((1+\mu)^{[a]}-1) = -p + \textstyle\sum_{a \in \FF_p} (1+\mu)^{[a]},\\
		\ptilde &:= \mu_0 + p = \textstyle\sum_{a \in \FF_p} (1+\mu)^{[a]}.
	\end{aligned}
\end{equation}
From loc.\ cit.\ recall that the element $\ptilde$ is the product of $[p]_q$ with a unit in $\AF$, and the element $\mu_0$ is the product of $\mu^{p-1}$ with a unit in $\AF^+$.
Now, we consider the subring $S := O_F\llbracket \mu_0 \rrbracket = (\AF^+)^{\FF_p^{\times}} \subset \AF^+$, which is stable under the action of $\varphi$ and $\Gamma_0$; we equip $S$ with the induced Frobenius endomorphism $\varphi$ and the induced continuous action of $\Gamma_0$.
Moreover, from loc.\ cit., for any $g$ in $\Gamma_0$, we have that $(g-1)\mu_0$ is an element of $\ptilde \mu_0 S$, and we may write $\varphi(\mu_0) = u \mu_0 \ptilde^{p-1}$, for some unit $u$ in $S$.
Furthermore, we have that $\mu_0 S = (\mu \AF^+)^{\FF_p^{\times}}$, and the natural $(\varphi, \Gamma_F)\equivariant$ injective map $S \rightarrow \AF^+$ is faithfully flat and finite of degree $p-1$.
Finally, we set $A_{F, 0} := S[1/\mu_0]^{\wedge}$, where ${}^{\wedge}$ denotes the $\padic$ completion, and note that $A_{F, 0}$ is equipped with an induced Frobenius endomorphism $\varphi$ and an induced continuous action of $\Gamma_0$.
In particular, we have $(\varphi, \Gamma_F)\equivariant$ homomorphism $A_{F,0} = \AF^{\FF_p^{\times}} \hookrightarrow \AF$.

\begin{rem}
	For consistency in notation, one should denote the ring $O_F\llbracket \mu_0 \rrbracket$ as $A_{F, 0}^+$.
	However, this ring shall appear frequently in the rest of the text and the aforementioned notation is too clunky.
	So we have chosen to denote the ring $O_F\llbracket \mu_0 \rrbracket$ simply by $S$.
\end{rem}

\begin{lem}\label{lem:sptildeinv_pid}
	The ring $S[1/\ptilde]$ is a principal ideal domain.
\end{lem}
\begin{proof}
	Note that $S$ is a two-dimensional regular local ring, so it is a unique factorisation domain (see \cite[\href{https://stacks.math.columbia.edu/tag/0FJH}{Tag 0FJH}]{stacks-project}).
	As $S[1/\ptilde]$ is a localisation of $S$, it is also a unique factorisation domain.
	Now, note that $\ptilde$ is in the maximal ideal $(p, \mu_0) \subset S$, so the Krull dimension of $S[1/\ptilde]$ is at most one.
	Moreover, $\mu_0S[1/\ptilde]$ is a non-zero prime ideal of $S[1/\ptilde]$, so the Krull dimension of $S[1/\ptilde]$ is at least one.
	In particular, we see that $S[1/\ptilde]$ is a one-dimensional unique factorisation domain, hence a principal ideal domain.
\end{proof}

\subsection{\texorpdfstring{$\padic$}{-} representations}\label{subsec:padicreps}

Let $T$ be a finite free $\ZZ_p\representation$ of $G_F$.
By the theory of \'etale $(\varphi, \Gamma_F)\modules$ (see \cite{fontaine-phigamma}), we may functorially associate to $T$ a finite free \'etale $(\varphi, \Gamma_F)\module$ $\DF(T)$ over $\AF$ of rank $= \rank_{\ZZ_p} T$.
Moreover, by taking $\Gamma_{\textrm{tor}}\textrm{-invariants}$ of $\DF(T)$, we obtain a finite free \'etale $(\varphi, \Gamma_0)\module$ $D_{F, 0}(T) := \DF(T)^{\Gamma_{\textrm{tor}}}$ over $A_{F, 0}$ of rank $= \rank_{\ZZ_p} T$.
These constructions are functorial in $\ZZ_p\textrm{-representations}$, and induce exact equivalences of $\otimes\textrm{-categories}$
\begin{equation}\label{eq:rep_phigamma}
	\begin{aligned}
		\Rep_{\ZZ_p}(G_F) &\isomorphic (\varphi, \Gamma_F)\Mod_{\AF}^{\etale},\\
		\Rep_{\ZZ_p}(G_F) &\isomorphic (\varphi, \Gamma_0)\Mod_{A_{F, 0}}^{\etale},
	\end{aligned}
\end{equation}
with an exact $\otimes\textrm{-compatible}$ quasi-inverse given as $\TF(D) := (W(\CC^{\flat}) \otimes_{\AF} D)^{\varphi=1}$ and $\TF(D_0) := (W(\CC^{\flat}) \otimes_{A_{F, 0}} D_0)^{\varphi=1}$, respectively.
Similar statements are also true for $\padic$ representations of $G_F$.

Next, from the $\padic$ Hodge theory of $G_F$ (see \cite{fontaine-reps}), one may attach to a $\padic$ representation $V$, a filtered $\varphi\module$ over $F$ of rank less than or equal to $\dim_{\QQ_p} V$, and given as $\Dcrys(V) := (\Bcrys(\OFbar) \otimes_{\QQ_p} V)^{G_F}$.
The representation $V$ is said to be \textit{crystalline} if the natural map $\Bcrys(\OFbar) \otimes_F \Dcrys(V) \rightarrow \Bcrys(\OFbar) \otimes_{\QQ_p} V$ is an isomorphism, or equivalently, $\dim_F \Dcrys(V) = \dim_{\QQ_p} V$.
Restricting $\Dcrys$ to the full subcategory of crystalline representations of $G_F$, and writing $\MF_F\wa(\varphi)$ for the category of weakly admissible filtered $\varphi\modules$ over $F$ (see \cite{fontaine-colmez}), we obtain an exact equivalence of $\otimes\textrm{-categories}$ (see \cite[Th\'eor\`eme A]{fontaine-colmez}):
\begin{equation}\label{eq:crys_equivalence}
	\Dcrys: \Rep_{\QQ_p}^{\crys}(G_F) \isomorphic \MF_F\wa(\varphi),
\end{equation}
with an exact $\otimes\textrm{-compatible}$ quasi-inverse given as $\Vcrys(D) := (\Fil^0(\Bcrys(\OFbar) \otimes_F D))^{\varphi=1}$.

\subsection{Galois cohomology and Fontaine--Herr complex}\label{subsec:fontaine_herr_complex}

Let $T$ be a $\ZZ_p\textrm{-representation}$ of $G_F$, and let $D := D_{F, 0}(T)$ denote the associated \'etale $(\varphi, \Gamma_F)\textrm{-module}$ over $A_{F, 0}$.
In \cite{herr-cohomologie}, Herr defined a three term complex in terms of $D$, which computes the continuous $G_F\cohomology$ of $T$.
More precisely, recall that $\gamma$ is a generator of $\Gamma_0$, and consider the following complex:
\begin{equation}\label{eq:fontaine_herr_complex}
	\calc^{\bullet}(D) \colon D \xrightarrow{\hspace{1mm}(\gamma-1, 1-\varphi)\hspace{1mm}} D \oplus D \xrightarrow{\hspace{3mm}(1-\varphi, \gamma-1)^{\intercal}\hspace{3mm}} D,
\end{equation}
where the first map is $x \mapsto ((\gamma-1)x, (1-\varphi)y)$, and the second map is $(x, y) \mapsto (1-\varphi)x - (\gamma-1)y$.
Then, the complex $\calc^{\bullet}(D)$ computes the continuous $G_F\cohomology$ of $T$ in each cohomological degree, i.e.\ for each $k \in \NN$, we have natural (in $T$) isomorphims $H^k(\calc^{\bullet}(D)) \isomorphic H^k_{\cont}(G_F, T)$.
In particular, it is clear that $H^k_{\cont}(G_F, T) = 0$, for $k \geqslant 3$.
To ease notations, from now onwards we shall write $H^k(G_F, T)$ instead of $H^k_{\cont}(G_F, T)$.

Note that for a $\ZZ_p\representation$ $T$ of $G_F$, the space $H^1(G_F, T)$ classifies all extension classes of $\ZZ_p$ by $T$ in the category of $\ZZ_p\textrm{-representations}$ of $G_F$.
Similarly, for an \'etale $(\varphi, \Gamma_0)\module$ $D$, the space $H^1(\calc^{\bullet}(D))$ classifies all extension classes of $A_{F, 0}$ by $D$ in the category of \'etale $(\varphi, \Gamma_0)\modules$ over $A_{F, 0}$.
Therefore, by the exact equivalence in \eqref{eq:rep_phigamma}, we have natural isomorphisms
\begin{equation*}
	H^1(G_F, T) \isomorphic \Ext^1_{\Rep_{\ZZ_p}(G_F)}(\ZZ_p, T) \isomorphic \Ext^1_{(\varphi, \Gamma_F)\Mod_{A_{F, 0}}^{\etale}}(A_{F, 0}, D) \lisomorphic H^1(\calc^{\bullet}(D)).
\end{equation*}

\subsection{Bloch--Kato Selmer groups}\label{subsec:blochkatoselmer_groups}

In this section, we will recall the definition of Bloch--Kato Selmer groups from \cite{bloch-kato}.
Let $V$ be a $\padic$ crystalline representation of $G_F$.
Then, we have a natural $G_F\equivariant$ morphism of $\QQ_p\textrm{-vector}$ spaces $V \rightarrow \Bcrys(\OFbar) \otimes_{\QQ_p} V$, sending $x \mapsto 1 \otimes x$.
By considering the continuous $G_F\cohomology$ groups, we obtain natural homomorphisms $H^k(G_F, V) \rightarrow H^k(G_F, \Bcrys(\OFbar) \otimes_{\QQ_p} V)$, for each $k \in \NN$.

\begin{defi}
	Define the \textit{Bloch--Kato Selmer groups} of $V$ denoted $H^k_f(G_F, V) \subset H^k(G_F, V)$, for $k \in \NN$, by setting
	\begin{equation*}
		H^k_f(G_F, V) := 
			\begin{cases}
				H^0(G_F, V), &\quad\textrm{if } k = 0\\
				\kert\big(H^1(G_F, V) \rightarrow H^1(G_F, \Bcrys(\OFbar) \otimes_{\QQ_p} V)\big), &\quad\textrm{if } k = 1\\
				0, &\quad\textrm{if } k \geqslant 2.
			\end{cases}
	\end{equation*}
\end{defi}

\begin{rem}\label{rem:h1f_extclass}
	For $k \in \NN$, the subspace $H^k_f(G_F, V) \subset H^k(G_F, V)$ are also referred to as the \textit{crystalline part of the Galois cohomology of $V$}.
	Notably, the subspace $H^1_f(G_F, V) \subset H^1(G_F, V)$ classifies all crystalline extension classes of $\QQ_p$ by $V$, i.e.\ we have natural isomorphisms
	\begin{equation*}
		H^1_f(G_F, V) \isomorphic \Ext^1_{\Rep_{\QQ_p}^{\crys}(G_F)}(\QQ_p, V) \isomorphic \Ext^1_{\MF_F\wa(\varphi)}(F, \Dcrys(V)),
	\end{equation*}
	where the last isomorphism follows from the exactness of $\Dcrys$ and $\Vcrys$ (see Section \ref{subsec:padicreps}).
\end{rem}

Note that we have a natural $\QQ_p\linear$ and $G_F\equivariant$ morphism $V \rightarrow \Fil^0 \Bcrys(\OFbar) \otimes_{\QQ_p} V$, sending $x \mapsto 1 \otimes x$, and it induces a natural homomorphism of abelian groups $H^1(G_F, V) \rightarrow H^1(G_F, \Fil^0\Bcrys(\OFbar) \otimes_{\QQ_p} V)$.
\begin{prop}\label{prop:hif_kerfil0}
	The following natural map is an isomorphism:
	\begin{equation*}
		\kert\big(H^1(G_F, V) \rightarrow H^1(G_F, \Fil^0 \Bcrys(\OFbar) \otimes_{\QQ_p} V)\big) \isomorphic H^1_f(G_F, V).
	\end{equation*}
\end{prop}
\begin{proof}	
	By the naturality of the action of $G_F$, we have the following commutative diagram:
	\begin{equation}\label{eq:h1_factors_fil0}
		\begin{tikzcd}
			H^1(G_F, V) \arrow[r] \arrow[rd] & H^1(G_F, \Fil^0 \Bcrys(\OFbar) \otimes_{\QQ_p} V) \arrow[d]\\
			& H^1(G_F, \Bcrys(\OFbar) \otimes_{\QQ_p} V).
		\end{tikzcd}
	\end{equation}
	To show the claim, it is enough to show that the right vertical arrow in \eqref{eq:h1_factors_fil0} is injective.
	Now, consider the following exact sequence:
	\begin{equation*}
		0 \longrightarrow \Fil^0 \Bcrys(\OFbar) \longrightarrow \Bcrys(\OFbar) \longrightarrow \Bcrys(\OFbar)/\Fil^0 \Bcrys(\OFbar) \longrightarrow 0.
	\end{equation*}
	Upon tensoring this exact sequence with $V$ and taking continuous $G_F\cohomology$, we obtain an injective map of $F\textrm{-vector spaces}$
	\begin{equation}\label{eq:dcrys_modfil0_inj}
		\Dcrys(V)/\Fil^0 \Dcrys(V) \longrightarrow \big(\Bcrys(\OFbar)/\Fil^0 \Bcrys(\OFbar) \otimes_{\QQ_p} V\big)^{G_F},
	\end{equation}
	and it is clear that the vertical map in \eqref{eq:h1_factors_fil0} is injective if and only if \eqref{eq:dcrys_modfil0_inj} is bijective.
	Now, as we have $\BdR(\OFbar) = \Fil^0 \BdR(\OFbar) + \Bcrys(\OFbar)^{\varphi=1}$ (see \cite[Proposition 1.17]{bloch-kato}), therefore, we obtain $G_F\equivariant$ isomorphisms
	\begin{equation*}
		\Bcrys(\OFbar)/\Fil^0 \Bcrys(\OFbar) \isomorphic \BdR(\OFbar)/\Fil^0 \BdR(\OFbar) \isomorphic \oplus_{k<0} \CC_p \cdot t^k,
	\end{equation*}
	where the last isomorphism follows from \cite[Section 1.5.5]{fontaine-corps-periodes}.
	Therefore, the codomain of the map \eqref{eq:dcrys_modfil0_inj} may be written as $\big(\Bcrys(\OFbar)/\Fil^0 \Bcrys(\OFbar) \otimes_{\QQ_p} V\big)^{G_F} = (\oplus_{k < 0} t^k \CC_p \otimes_{\QQ_p} V)^{G_F} = \oplus_{k<0} \gr^k \Dcrys(V)$.
	Counting dimensions, we note that we have 
	\begin{equation*}
		\dim_F(\Fil^0 \Dcrys(V)) + \dim_F(\oplus_{k<0} \gr^k \Dcrys(V)) = \dim_F \Dcrys(V),
	\end{equation*}
	so the domain and the codomain of the $F\linear$ injective map in \eqref{eq:dcrys_modfil0_inj} have the same dimension.
	Hence, \eqref{eq:dcrys_modfil0_inj} is bijective, thus allowing us to conclude.
\end{proof}

The following claim can also be obtained as a consequence of \cite[Corollary 3.8.4]{bloch-kato}.
\begin{cor}\label{cor:dcrys_complex_bks}
	Let $V$ be a $\padic$ crystalline representation of $G_F$.
	Then, the following complex:
	\begin{equation}\label{eq:dcrys_complex_bks}
		\cald^{\bullet}(\Dcrys(V)) \colon \Fil^0 \Dcrys(V) \xrightarrow{\hspace{1mm} 1-\varphi \hspace{1mm}} \Dcrys(V),
	\end{equation}
	computes the crystalline part of the Galois cohomology of $V$, i.e.\ $H^k(\cald^{\bullet}(\Dcrys(V))) \isomorphic H^k_f(G_F, V)$, for each $k \in \NN$.
\end{cor}
\begin{proof}
	Tensoring the fundamental exact sequence in \eqref{eq:fes_bcris} with $V$, we obtain a $G_F\equivariant$ exact sequence
	\begin{equation*}
		0 \longrightarrow V \longrightarrow \Fil^0\Bcrys(\OFbar) \otimes_{\QQ_p} V \xrightarrow{\hspace{1mm} 1-\varphi \hspace{1mm}} \Bcrys(\OFbar) \otimes_{\QQ_p} V \longrightarrow 0.
	\end{equation*}
	By computing the continuous Galois cohomoloy, we obtain the following long exact sequence:
	\begin{equation*}
		\begin{aligned}
			0 \longrightarrow H^0(G_F, V) \longrightarrow \Fil^0 \Dcrys(V) &\xrightarrow{\hspace{1mm} 1-\varphi \hspace{1mm}} \Dcrys(V) \longrightarrow H^1(G_F, V) \longrightarrow \\
			&\longrightarrow H^1(G_F, \Fil^0 \Bcrys(\OFbar) \otimes_{\QQ_p} V).
		\end{aligned}
	\end{equation*}
	The claim now follows from Proposition \ref{prop:hif_kerfil0}.
\end{proof}

\section{Wach modules}\label{sec:wachmods}

In this section, we will recall the definition of Wach modules, their relationship with $\padic$ crystalline representations, and prove some results on the Nygaard filtration on Wach modules.
From Section \ref{subsec:period_rings} recall that we have the ring $\AF^+ = O_F\llbracket \mu \rrbracket$ equipped with a Frobenius endomorphism $\varphi$ and a continuous action of $\Gamma_F$.
Moreover, we fixed $[p]_q = \varphi(\mu)/\mu$ in $\AF^+$.

\begin{defi}\label{defi:wachmod_af+}
	Let $a, b \in \ZZ$ with $b \geqslant a$.
	A \textit{Wach module} over $\AF^+$ with weights in the interval $[a, b]$ is a finite free $\AF^+\module$ $N$ equipped with a continuous and semilinear action of $\Gamma_F$ satisfying the following:
	\begin{enumarabicup}
	\item $\Gamma_F$ acts trivially on $N/\mu N$.

	\item There is a Frobenius-semilinear operator $\varphi: N[1/\mu] \rightarrow N[1/\varphi(\mu)]$ commuting with the action of $\Gamma_F$ such that $\varphi(\mu^b N) \subset \mu^b N$, and the cokernel of the induced injective map $(1 \otimes \varphi) \colon \varphi^{\ast}(\mu^b N) \rightarrow \mu^b N$ is killed by $[p]_q^{b-a}$.
	\end{enumarabicup}
	Say that $N$ is \textit{effective} if one may take $b = 0$ and $a \leqslant 0$.
	Denote the category of Wach modules over $\AF^+$ as $(\varphi, \Gamma)\Mod_{\AF^+}^{[p]_q}$ with morphisms between objects being $\AF^+\linear$, $\Gamma_F\equivariant$ and $\varphi\equivariant$ (after inverting $\mu$) morphisms.
\end{defi}

\begin{rem}\label{rem:finiteheight_equiv}
	Let $N$ be a finitely generated $\AF^+\module$.
	Then, from \cite[Lemma 3.10]{abhinandan-relative-wach-ii} we have that the condition (2) of Definition \ref{defi:wachmod_af+} is equivalent to giving an $\AF^+\linear$ and $\Gamma_F\equivariant$ isomorphism $\varphi_N \colon (\varphi^*N)[1/[p]_q] = (\AF^+ \otimes_{\varphi, \AF^+} N)[1/[p]_q] \isomorphic N[1/[p]_q]$.
\end{rem}

\begin{rem}\label{rem:wachtoetale_af+}
	Extending scalars along $\AF^+ \rightarrow \AF$ induces a fully faithful functor $(\varphi, \Gamma_F)\Mod_{\AF^+}^{[p]_q} \rightarrow (\varphi, \Gamma_F)\Mod_{\AF}^{\etale}$ (see \cite[Proposition 3.3]{abhinandan-imperfect-wach}).
	Moreover, the preceding functor preserves short exact sequences because the map $\AF^+ \rightarrow \AF$ is flat.
\end{rem}

\begin{nota}
	We shall say that a $(\varphi, \Gamma_F)\equivariant$ sequence of Wach modules over $\AF^+$,
	\begin{equation*}
		0 \longrightarrow N_1 \longrightarrow N_2 \longrightarrow N_3 \longrightarrow 0,
	\end{equation*}
	is \textit{exact} if it is exact as a sequence of modules over $\AF^+$.
\end{nota}

\subsection{Descent of Wach modules}\label{eq:wachmod_descent}

In this section, we shall show that Wach modules over $\AF^+$ descend to a certain subring of $\AF^+$ (see \cite{fontaine-phigamma, wach-torsion}).
From Section \ref{subsec:period_rings} recall that we have the ring $S = O_F\llbracket \mu_0 \rrbracket$, equipped with a Frobenius endomorphism $\varphi$ and a continuous action of $\Gamma_0$.
Moreover, have that $\mu_0 = \sum_{a \in \FF_p^{\times}} ((1+\mu)^{[a]}-1)$, and we fixed $\ptilde = \mu_0 + p$ in $S$.

\begin{defi}\label{defi:wachmod_s}
	Let $a, b \in \ZZ$ with $b \geqslant a$.
	A \textit{Wach module} over $S$ with weights in the interval $[a, b]$ is a finite free $S\module$ $M$ equipped with a continuous and semilinear action of $\Gamma_0$ satisfying the following:
	\begin{enumarabicup}
	\item $\Gamma_0$ acts trivially on $M/\mu_0 M$.

	\item There is a Frobenius-semilinear operator $\varphi: M[1/\mu_0] \rightarrow M[1/\varphi(\mu_0)]$ commuting with the action of $\Gamma_0$ such that $\varphi(\mu_0^b M) \subset \mu_0^b M$, and the cokernel of the induced injective map $(1 \otimes \varphi) \colon \varphi^{\ast}(\mu_0^b M) \rightarrow \mu_0^b M$ is killed by $\ptilde^{b-a}$.
	\end{enumarabicup}
	Say that $M$ is \textit{effective} if one may take $b = 0$ and $a \leqslant 0$.
	Denote the category of Wach modules over $S$ as $(\varphi, \Gamma)\Mod_{S}^{\ptilde}$ with morphisms between objects being $S\linear$, $\Gamma_0\equivariant$ and $\varphi\equivariant$ (after inverting $\mu_0$) morphisms.
\end{defi}

\begin{nota}
	We shall say that a $(\varphi, \Gamma_0)\equivariant$ sequence of Wach modules over $S$,
	\begin{equation*}
		0 \longrightarrow M_1 \longrightarrow M_2 \longrightarrow M_3 \longrightarrow 0,
	\end{equation*}
	is \textit{exact} if it is exact as a sequence of modules over $S$.
\end{nota}

\begin{lem}\label{lem:finiteheight_equiv}
	Let $M$ be a finite free $S\module$ satisfying condition (1) of Definition \ref{defi:wachmod_s}.
	Then, condition (2) of Definition \ref{defi:wachmod_s} is equivalent to giving an $S\linear$ and $\Gamma_0\equivariant$ isomorphism $\varphi_M \colon (\varphi^*M)[1/\ptilde] = (S \otimes_{\varphi, S} M)[1/\ptilde] \isomorphic M[1/\ptilde]$.
\end{lem}
\begin{proof}
	Suppose that $M$ satisfies condition (2) of Definition \ref{defi:wachmod_s}.
	Then, the map $1 \otimes \varphi \colon \varphi^*(\mu_0^b M) \rightarrow \mu_0^b M$ induces an isomorphism $1 \otimes \varphi \colon (\mu_0^b \varphi^*M)[1/\ptilde] \isomorphic (\mu_0^b M)[1/\ptilde]$.
	Hence, we obtain an isomorphism
	\begin{equation*}
		\varphi_M \colon (\varphi^*M)[1/\ptilde] \xrightarrow[\sim]{\hspace{1mm} \mu_0^b \hspace{1mm}} (\mu_0^b \varphi^*M)[1/\ptilde] \xrightarrow[\sim]{\hspace{1mm} 1 \otimes \varphi \hspace{1mm}} (\mu_0^b M)[1/\ptilde] \xleftarrow[\sim]{\hspace{1mm} \mu_0^b \hspace{1mm}} M[1/\ptilde].
	\end{equation*}
	Since, $1 \otimes \varphi$ commutes with the action of $\Gamma_0$, we deduce that $\varphi_M$ is $\Gamma_0\equivariant$.

	Conversely, let us suppose that $M$ is equipped with an $S\linear$ and $\Gamma_0\equivariant$ isomorphism $\varphi_M \colon (\varphi^*M)[1/\ptilde] \isomorphic M[1/\ptilde]$.
	Then, note that for some $a, b \in \ZZ$ with $b \geqslant a$ we may write $\ptilde^b \varphi_M(\varphi^* M) \subset M \subset \ptilde^a \varphi_M(\varphi^*M)$.
	So we get an $S\textrm{-semilinear}$ and $\Gamma_0\equivariant$ map as the composition $\varphi \colon \mu_0^b M \xrightarrow{\hspace{1mm} \textrm{can} \hspace{1mm}} \varphi^*(\mu_0^b M) \xrightarrow{\hspace{1mm} \varphi_M \hspace{1mm}} \mu_0^b M$.
	This extends to an $S\textrm{-semilinear}$ and $\Gamma_0\equivariant$ map $\varphi \colon M[1/\mu_0] \rightarrow M[1/\varphi(\mu_0)]$, and we have that
	\begin{equation*}
		\varphi_M(\varphi^*(\mu_0^b M)) = \mu_0^b \ptilde^{(p-1)b} \varphi_M(\varphi^*M) \subset \mu_0^b \ptilde^{(p-1)(b-1)} M \subset \ptilde^{a-b} \varphi_M(\varphi^*(\mu_0^b M)),
	\end{equation*}
	where we used that $\varphi(\mu_0) = u \mu_0\ptilde^{p-1}$, for some unit $u$ in $S$ (see the discussion after \eqref{eq:pq_mu0_ptilde_defi}).
	Then, it follows that $1 \otimes \varphi = \varphi_M \colon \varphi^*(\mu_0^b M) \rightarrow \mu_0^b M$ is injective, its cokernel is killed by $\ptilde^{b-a}$, and it commutes with the action of $\Gamma_0$.
	Hence, $M$ satisfies condition (2) of Definition \ref{defi:wachmod_s}.
\end{proof}

Next, let us compare the two notions of Wach modules over $S$ and $\AF^+$, respectively.
We start with the following observation:
\begin{prop}\label{prop:wachmod_fpx_descent}
	Let $N$ be a Wach module over $\AF^+$.
	Then $M := N^{\FF_p^{\times}}$ is a Wach module over $S$, and the $\AF^+\textrm{-linear}$ extension of the $S\linear$ inclusion $M \subset N$, induces a natural isomorphism 
	\begin{equation}\label{eq:af+m_is_n}
		\AF^+ \otimes_S M \isomorphic N,
	\end{equation}
	of Wach modules over $\AF^+$.
	Moreover, the isomorphism in \eqref{eq:af+m_is_n} induces a natural isomorphism of $O_F\modules$ $M/\mu_0 M \isomorphic N/\mu N$ compatible with the respective Frobenii.
\end{prop}
\begin{proof}
	The claim follows from \cite[Theorem 1.5]{abhinandan-prismatic-wach}.
	However, we provide another self-contained proof.
	Let $A := \AF^+$, and from Remark \ref{rem:fpx_decomp} note that for the $A\module$ $N$ there exists a natural $\FF_p^{\times}\textrm{-decomposition}$ $N = \bigoplus_{i=0}^{p-2} N_i$, where $M = N^{\FF_p^{\times}} = N_0$, and each $N_i$ is a $(p, \mu_0)\textrm{-adically}$ complete $S\module$ equipped with a continuous and semilinear action of $\Gamma_0$.
	Moreover, recall that $A$ is flat and finite of degree $p-1$ over the noetherian ring $S$, so it follows that $N$ is finite free over $S$, and the $S\textrm{-submodule}$ $M \subset N$ is finitely generated.
	Now, let us consider the following natural commutative diagram:
	\begin{equation}\label{eq:af+m_in_n}
		\begin{tikzcd}
			A \otimes_S M \arrow[r] \arrow[d] & N \arrow[d]\\
			A \otimes_S M[1/\ptilde] \arrow[r] & N[1/\ptilde],
		\end{tikzcd}
	\end{equation}
	where the right vertical arrow is the natural inclusion, the left vertical arrow is injective because $M$ is $\ptilde\textrm{-torsion free}$ (since the same is true for $N$) and the natural map $S \rightarrow A$ is flat.

	We claim that the top horizontal arrow of \eqref{eq:af+m_in_n} is bijective.
	First, note that to show the injectivity of the top horizontal arrow, it is enough to show that the injectivity of the bottom horizontal arrow in \eqref{eq:af+m_in_n}.
	The module $N[1/\ptilde]$ is finite free over $A[1/\ptilde]$, and the map $S \rightarrow A$ is faithfully flat and finite of degree $p-1$, so we see that $N[1/\ptilde]$ is finite free over $S[1/\ptilde]$ as well.
	As $N[1/\ptilde]$ is equipped with an induced action of $\Gamma_F$, therefore, from the decomposition in Remark \ref{rem:fpx_decomp}, it follows that $M[1/\ptilde] \isomorphic (N[1/\ptilde])_0$ as $S[1/\ptilde]\modules$, in particular, $M[1/\ptilde]$ is finite projective over $S[1/\ptilde]$, hence finite free because $S[1/\ptilde]$ is a principal ideal domain (see Lemma \ref{lem:sptildeinv_pid}).
	As the natural map $S \rightarrow A$ is flat, it follows that the bottom horizontal arrow in \eqref{eq:af+m_in_n} is injective.

	Next, let us show that the top horizontal arrow in \eqref{eq:af+m_in_n} is surjective.
	Note that we have an $O_F\linear$ and $\Gamma_F\equivariant$ surjection $N \rightarrow N/\mu N$, and using the decomposition in Remark \ref{rem:fpx_decomp}, it may be rewritten as a $\Gamma_F\equivariant$ surjection $\bigoplus_{i=0}^{p-2} N_i \rightarrow \bigoplus_{i=0}^{p-2} (N/\mu N)_i$.
	In particular, the latter map is surjective on each term, i.e.\ the induced $O_F\linear$ map $N_i \rightarrow (N/\mu N)_i$ is surjective, for each $0 \leqslant i \leqslant p-2$.
	However, recall that $\Gamma_F$ acts trivially on $N/\mu N$, therefore, we see that $N/\mu N = \bigoplus_{i=0}^{p-2} (N/\mu N)_i = (N/\mu N)_0$.
	In particular, the natural $O_F\linear$ map $M \rightarrow N/\mu N$ is surjective.
	Since $\mu$ belongs to the Jacobson radical of $A$, therefore, Nakayama Lemma implies that the natural $(\varphi, \Gamma_F)\equivariant$ map $A \otimes_S M \rightarrow N$ is also surjective.
	Hence, it follows that the $A\linear$ extension of the $S\linear$ inclusion $M \subset N$ induces the natural $(\varphi, \Gamma_F)\equivariant$ isomorphism in \eqref{eq:af+m_is_n}, i.e.\ $A \otimes_S M \isomorphic N$.

	Let us now compute the kernel of the surjective map $M \rightarrow N/\mu N$, which is given as $\mu N \cap M$ inside $N$.
	Using the decomposition in Remark \ref{rem:fpx_decomp}, we get that $\mu N \cap M = (\mu M)_0$, and we claim that $\mu_0 M \isomorphic (\mu N)_0$ as $S\modules$.
	As the natural map $\mu_0 M \rightarrow (\mu N)_0$ is injective, we need to show that it is surjective as well.
	Note that we may write
	\begin{equation*}
		(\mu N)_0 = \textstyle\sum_{i+j = 0 \textrm{ mod } p-1} (\mu A)_i \otimes_S N_j.
	\end{equation*}
	Moreover, using the $(\varphi, \Gamma_F)\equivariant$ isomorphism $A \otimes_S M \isomorphic N$, note that we have $A_j \otimes_S M \isomorphic N_j$, for each $0 \leqslant j \leqslant p-2$.
	Therefore, for each pair $i, j$ such that $i+j = 0 \textrm{ mod } p-1$, we see that 
	\begin{equation*}
		(\mu A)_i \otimes_S N_j \lisomorphic (\mu A)_i \otimes_S A_j \otimes_S N_0 \subset (\mu A)_0 \otimes_S N_0.
	\end{equation*}
	Now, recall that we have, $\mu_0 S = (\mu A)^{\FF_p^{\times}} = (\mu A)_0$.
	Hence, from the preceding discussion, we get that $\mu_0 M \isomorphic (\mu N)_0$.
	Subsequently, we have also obtained that $M/\mu_0 M \isomorphic N/\mu N$ as $O_F\modules$, and therefore, the action of $\Gamma_0$ is trivial on $M/\mu_0 M$.
	Moreover, note that the $S\module$ $M$ is $p\torsion$ free and $\mu_0\torsion$ free, and $M/\mu_0 M \isomorphic N/\mu N$ is $p\torsion$ free.
	Therefore, from \cite[Lemma 3.5]{abhinandan-relative-wach-ii} and \cite[Proposition B.1.2.4]{fontaine-phigamma}, it follows that $M$ is finite free over $A$.

	Finally, let us show the Frobenius condition on $M$ (see Definition \ref{defi:wachmod_s} and Lemma \ref{lem:finiteheight_equiv}).
	From Section \ref{subsec:period_rings}, recall that $\ptilde$ is the product of $[p]_q$ with a unit in $A$, and so the Frobenius structure on $N$ may also be described as an $A\linear$ isomorphism $\varphi^*(N)[1/\ptilde] \isomorphic N[1/\ptilde]$ (see Definition \ref{defi:wachmod_af+} and Remark \ref{rem:finiteheight_equiv}).
	Moreover, the Frobenius on $N$ commutes with the action of $\Gamma_F$.
	Therefore, by taking invariants under the action of $\FF_p^{\times}$, of the isomorphism $\varphi^*(N)[1/\ptilde] \isomorphic N[1/\ptilde]$, and using the $(\varphi, \Gamma_F)\equivariant$ isomorphism in \eqref{eq:af+m_is_n}, we conclude that $M$ is equipped with an $S\linear$ isomorphism $\varphi^*(M)[1/\ptilde] \isomorphic M[1/\ptilde]$, compatible with the natural action of $\Gamma_0$.
	This allows us to conclude.
\end{proof}

Now, let $M$ be a Wach module over $S$, and let $N := \AF^+ \otimes_S M$ be a finite free module over $\AF^+$.
Then using the natural $(\varphi, \Gamma_0)\action$ on $M$, we see that $N$ is naturally equipped with a semilinear and continuous action of $\Gamma_F$, and an $\AF^+\linear$ isomorphism $\varphi_N \colon (\varphi^*N)[1/\ptilde] \isomorphic N[1/\ptilde]$.
As $\ptilde$ is the product of $[p]_q$ with a unit in $\AF^+$ (see Section \ref{subsec:period_rings}), therefore, using Remark \ref{rem:finiteheight_equiv} we conclude that $N$ is a Wach module over $\AF^+$.
More generally, we have that,
\begin{thm}\label{thm:wachmod_s_af+_equiv}
	The following natural functor induces an equivalence of $\otimes\textrm{-categories}$
	\begin{equation}\label{eq:wachmod_s_af+_equiv}
		\begin{aligned}
			(\varphi, \Gamma_0)\Mod_{S}^{\ptilde} &\isomorphic (\varphi, \Gamma_F)\Mod_{\AF^+}^{[p]_q}\\
			M &\longmapsto \AF^+ \otimes_S M,
		\end{aligned}
	\end{equation}
	with a $\otimes\textrm{-compatible}$ quasi-inverse functor given as $N \mapsto N^{\FF_p^{\times}}$.
	Moreover, the functor in \eqref{eq:wachmod_s_af+_equiv} and its quasi-inverse preserve short exact sequences.
\end{thm}
In the statement above, short exact sequences should be interpreted in the sense of the notation introduced after Remark \ref{rem:wachtoetale_af+} and Definition \ref{defi:wachmod_s}.
\begin{proof}
	Note that the natural $(\varphi, \Gamma_F)\equivariant$ map $S \rightarrow \AF^+$ is faithfully flat and finite of degree $p-1$, so it follows the functor in \eqref{eq:wachmod_s_af+_equiv} fully faithful and preserves short exact sequences. 
	Moreover, by Proposition \ref{prop:wachmod_fpx_descent} we see that \eqref{eq:wachmod_s_af+_equiv} is essentially surjective, and its compatibility with tensor products is obvious.
	It remains to show that the quasi-inverse functor is compatible with $\otimes\textrm{-products}$ and preserves short exact sequences.
	To check the compatibility with tensor products, let $N_1$ and $N_2$ be two Wach modules over $\AF^+$, set $M_i := N_i^{\FF_p^{\times}}$, for each $i = 1, 2$, and using Proposition \ref{prop:wachmod_fpx_descent}, note that we have 
	\begin{equation*}
		(N_1 \otimes_{\AF^+} N_2)^{\FF_p^{\times}} = ((\AF^+ \otimes_S M_1) \otimes_{\AF^+} (\AF^+ \otimes_S M_2))^{\FF_p^{\times}} = M_1 \otimes_S M_2.
	\end{equation*}
	Next, consider an $\AF^+\linear$ and $(\varphi, \Gamma_F)\equivariant$ exact sequence of Wach modules over $\AF^+$ as
	\begin{equation*}
		0 \longrightarrow N_1 \longrightarrow N_2 \longrightarrow N_3 \longrightarrow 0.
	\end{equation*}
	Setting $M_i := N_i^{\FF_p^{\times}}$, for each $i = 1, 2, 3$, and considering the associated long exact exact sequence for the cohomology of the $\FF_p^{\times}\textrm{-action}$, and noting that $H^1(\FF_p^{\times}, N_1) = 0$, since $p-1$ is invertible in $\ZZ_p$, we obtain an $S\linear$ and $(\varphi, \Gamma_F)\equivariant$ exact sequence of Wach modules over $S$ as
	\begin{equation*}
		0 \longrightarrow M_1 \longrightarrow M_2 \longrightarrow M_3 \longrightarrow 0.
	\end{equation*}
	This allows us to conclude.
\end{proof}

\begin{rem}\label{rem:wachtoetale_s}
	Note that extension of scalars along $S \rightarrow A_{F, 0}$ induces a fully faithful functor $(\varphi, \Gamma_0)\Mod_{S}^{\ptilde} \rightarrow (\varphi, \Gamma_0)\Mod_{A_{F, 0}}^{\etale}$.
	Indeed, observe that we have the following bijection of sets:
	\begin{align*}
		\Hom_{(\varphi, \Gamma)\Mod_{S}^{\ptilde}}(M, M') &\isomorphic \Hom_{(\varphi, \Gamma)\Mod_{\AF^+}^{[p]_q}}(\AF^+ \otimes_{S} M, \AF^+ \otimes_{S} M')\\
			&\isomorphic \Hom_{(\varphi, \Gamma)\Mod_{\AF}^{\etale}}(\AF \otimes_{S} M, \AF \otimes_{S} M')\\
			&\lisomorphic \Hom_{(\varphi, \Gamma)\Mod_{A_{F,0}}^{\etale}}(A_{F,0} \otimes_{S} M, A_{F,0} \otimes_{S} M'),
	\end{align*}
	where the first bijection follows from Theorem \ref{thm:wachmod_s_af+_equiv}, the second bijection follows from Remark \ref{rem:wachtoetale_af+}, and the last bijection follows from \eqref{eq:rep_phigamma}.
	Moreover, the preceding functor preserves short exact sequences because the map $S \rightarrow A_{F,0}$ is flat.
\end{rem}

\subsection{Wach modules and crystalline representations}\label{subsec:wachmods_to_crysreps}

Let $\Rep_{\ZZ_p}^{\crys}(G_F)$ denote the category of $\ZZ_p\textrm{-lattices}$ inside $\padic$ crystalline representations of $G_F$.
To any $T$ in $\Rep_{\ZZ_p}^{\crys}(G_F)$, Berger functorially associated a unique Wach module $\NF(T)$ over $\AF^+$ in \cite{berger-limites}.
More generally, we have the following:
\begin{thm}[{\cite{fontaine-phigamma, wach-free, colmez-hauteur, berger-limites}}]\label{thm:crystalline_wach_af+_equivalence}
	The Wach module functor induces an equivalence of $\otimes\textrm{-catgeories}$
	\begin{equation*}
		\Rep_{\ZZ_p}^{\crys}(G_F) \isomorphic (\varphi, \Gamma_F)\Mod_{\AF^+}^{[p]_q}, \hspace{5mm} T \longmapsto \NF(T),
	\end{equation*}
	with a $\otimes\textrm{-compatible}$ quasi-inverse given as $N \mapsto \TF(N) = \big(W(\CC_p^{\flat}) \otimes_{\AF^+} N\big)^{\varphi=1}$.
\end{thm}

Let us recall an important comparison result from \cite{berger-limites}, between a Wach module and its associated $\ZZ_p\textrm{-representation}$, which will be used in Lemma \ref{lem:h0_syn_blochkatoselmer} below.
To recall the result, we need to introduce some notations from op.\ cit.
Let $A$ denote the $\padic$ completion of the maximal unramified extension of $\AF$ inside $W(\CC_p^{\flat})$.
The ring $A$ is stable under the $(\varphi, G_F)\action$ on $W(\CC_p^{\flat})$, and we equip it with induced structures.
Then, we have that $A^{H_F} = \AF$ and $A^{\varphi = 1} = \ZZ_p$.
Moreover, we set $A^+ := \Ainf(\OFbar) \cap A \subset W(\CC_p^{\flat})$, which is stable under the $(\varphi, G_F)\action$ on $A$, and we have that $(A^+)^{H_F} = \AF^+$ and $(A^+)^{\varphi = 1} = \ZZ_p$.
\begin{prop}[{\cite[Th\'eor\`eme III.3.1]{berger-limites}}]\label{prop:wachmod_rep_comp}
	Let $N$ be a Wach module over $\AF^+$ and $\TF(N)$ its associated $\ZZ_p\textrm{-representation}$.
	Then, we have a natural $A^+[1/\mu]\linear$ and $(\varphi, G_F)\equivariant$ comparison isomorphism $A^+[1/\mu] \otimes_{\AF^+} N \isomorphic A^+[1/\mu] \otimes_{\ZZ_p} \TF(N)$.
\end{prop}

\begin{rem}\label{rem:crystalline_wach_equivalence_rational}
	In Theorem \ref{thm:crystalline_wach_af+_equivalence}, note that the functor $\NF$ is not exact, i.e.\ it does not preserve short exact sequences (see \cite[Example 7.1]{chang-diamond}).
	However, by passing to the associated isogeny categories, the Wach module functor induces an exact equivalence of $\otimes\textrm{-categories}$ $\Rep_{\QQ_p}^{\crys}(G_F) \isomorphic (\varphi, \Gamma)\Mod_{\AF^+[1/p]}^{[p]_q}$, via $T[1/p] \mapsto \NF(T)[1/p]$, and with an exact $\otimes\textrm{-compatible}$ quasi-inverse given as $N[1/p] \mapsto \VF(N[1/p]) = \big(W(\CC_p^{\flat}) \otimes_{\AF^+} N[1/p]\big)^{\varphi=1}$ (see \cite[Corollary 4.3]{abhinandan-imperfect-wach}).
\end{rem}

\begin{rem}\label{rem:wachmod_cycltwist}
	Let $N$ be a Wach module over $\AF^+$ and $T = \TF(N)$ the associated $\ZZ_p\representation$ of $G_F$ from Theorem \ref{thm:crystalline_wach_af+_equivalence}.
	Then, for each $r \in \ZZ$, it is straightforward to verify that $\mu^{-r} N(r)$ is a Wach module over $\AF^+$ and $\TF(\mu^{-r} N(r)) \isomorphic T(r)$, where $(r)$ denotes a twist by $\chi^r$ for the Wach module and the representation.
	For many more interesting and explicit examples, we refer the reader to \cite[Appendice A]{berger-limites} and \cite{berger-li-zhu}.
\end{rem}

By combining Theorem \ref{thm:wachmod_s_af+_equiv} and Theorem \ref{thm:crystalline_wach_af+_equivalence}, we obtain the following:
\begin{thm}\label{thm:crystalline_wach_s_equivalence}
	The following functor induces an equivalence of $\otimes\textrm{-catgeories}$
	\begin{equation*}
		\Rep_{\ZZ_p}^{\crys}(G_F) \isomorphic (\varphi, \Gamma_0)\Mod_S^{\ptilde}, \hspace{5mm} T \longmapsto M_F(T) := \NF(T)^{\FF_p^{\times}},
	\end{equation*}
	with a $\otimes\textrm{-compatible}$ quasi-inverse given as $M \mapsto \TF(M) = \big(W(\CC_p^{\flat}) \otimes_S M\big)^{\varphi=1}$.
\end{thm}

\subsection{Nygaard filtration on Wach modules}\label{subsec:nyfil}

In this section, we will study the Nygaard filtration on Wach modules over $\AF^+$ and $S$, respectively.
We begin with the former case.

\subsubsection{Filtration on Wach modules over \texorpdfstring{$\AF^+$}{-}}\label{subsubsec:nyfil_wachmod_af+}

Let $N$ be a Wach module over $\AF^+$.
We equip $N$ with a decreasing filtration called the \textit{Nygaard filtration} as
\begin{equation}\label{eq:nyfil_n}
	\Fil^k N := \{ x \in N \textrm{ such that } \varphi(x) \in [p]_q^k N\}, \textrm{ for } k \in \ZZ.
\end{equation}
From the definition it is clear that $N$ is effective if and only if $\Fil^0 N = N$.
Similarly, we equip the $\AF^+[1/p]\module$ $N[1/p]$ with a Nygaard filtration, and it is easy to see that $\Fil^k(N[1/p]) = (\Fil^k N)[1/p]$.

\begin{lem}[{\cite[Lemma 3.3 and Lemma 3.4]{abhinandan-syntomic}}]\label{lem:wachmod_af+_twist_fil}
	Let $N$ be a Wach module over $\AF^+$.
	\begin{enumarabicup}
		\item For any $k, r \in \ZZ$ and the Wach module $\mu^{-r}N(r)$ over $\AF^+$, we have that $\Fil^k(\mu^{-r}N(r)) = \mu^{-r} (\Fil^{r+k} N) (r)$.

		\item For all $j, k \in \ZZ$, we have that $\mu^{-j} \Fil^k N \cap \mu^{-j+1} N = \mu^{-j+1}\Fil^{k-1} N \subset N[1/\mu]$.
			A similar statement is true for the $\AF^+[1/p]\module$ $N[1/p]$.
	\end{enumarabicup}
\end{lem}
\begin{proof}
	For the claim (1), note that the inclusion $\mu^{-r} (\Fil^{r+k} N) (r) \subset \Fil^k(\mu^{-r}N(r))$ is obvious.
	For the converse, let $\mu^{-r} x \otimes \epsilon^{\otimes r}$ be an element of $\Fil^k (\mu^{-r} N(r))$, with $x$ in $N$ and $\epsilon^{\otimes r}$ an $\AF^+\textrm{-basis}$ of $\AF^+(r)$.
	By assumption, we have that $\varphi(\mu^{-r} x \otimes \epsilon^{\otimes r}) = ([p]_q \mu)^{-r} \varphi(x) \otimes \epsilon^{\otimes r}$ belongs to $[p]_q^k \mu^{-r} N(r)$.
	Therefore, we see that $\varphi(x)$ belongs to $[p]_q^{r+k} N$, i.e.\ $x$ is in $\Fil^{r+k} N$.

	For (2), note that it is enough to show the claim for $j = 0$, i.e.\ $\Fil^k N \cap \mu N = \mu \Fil^{k-1} N \subset N$.
	Now, using (1) we may assume that $N$ is effective.
	The claim is obvious if $\Fil^{k-1} N = N$.
	So, we may further assume that $\Fil^{k-1} N \subsetneq N$, i.e.\ $k \geqslant 2$.
	Let $x$ be an element of $\Fil^k N \cap \mu N$, and write $x = \mu y$ for some $y$ in $N$.
	We claim that $y$ is in $\Fil^{k-1} N$.
	Indeed, note that $\varphi(x)$ is in $[p]_q^k N$, therefore, we get that $\mu \varphi(y)$ is in $[p]_q^{k-1} N$, i.e.\ $\mu \varphi(y) = [p]_q^{k-1} z$, for some $z$ in $N$.
	In particular, $[p]_q^{r-1}z = p^{r-1}z = 0 \textrm{ mod } \mu N$.
	But, we have that $N/\mu N$ is $p\textrm{-torsion}$ free, so it follows that $z = 0 \textrm{ mod } \mu N$, i.e.\ $y$ belongs to $\Fil^{k-1} N$.
	The other inclusion is obvious, as we have that $\mu \Fil^{k-1} N \subset \Fil^k N$.
	This concludes our proof.
\end{proof}

Next, we note that $(N/\mu N)[1/p]$ is a $\varphi\module$ over $F$ since $[p]_q = p \textmod \mu \AF^+$, and $N/\mu N$ is equipped with a filtration $\Fil^k(N/\mu N)$ given as the image of $\Fil^k N$ under the surjection $N \twoheadrightarrow N/\mu N$.
We equip $(N/\mu N)[1/p]$ with the induced filtration $\Fil^k((N/\mu N)[1/p]) := \Fil^k(N/\mu N)[1/p]$, and note that it is a filtered $\varphi\module$ over $F$.
From \cite[Th\'eor\`eme III.4.4]{berger-limites} and \cite[Theorem 1.7 \& Remark 1.8]{abhinandan-imperfect-wach} we have the following:
\begin{thm}\label{thm:qdeformation_dcrys}
	Let $N$ be a Wach module over $\AF^+$ and $V := \TF(N)[1/p]$ the associated crystalline representation of $G_F$ from Theorem \ref{thm:crystalline_wach_af+_equivalence}.
	Then, we have that $(N/\mu N)[1/p] \isomorphic \Dcrys(V)$ as filtered $\varphi\modules$ over $F$.
\end{thm}

From Theorem \ref{thm:qdeformation_dcrys} we have a surjection $\Fil^k N[1/p] \twoheadrightarrow \Fil^k \Dcrys(V)$, and we would like to determine its kernel.
\begin{lem}\label{lem:fil_nmodmu_ker}
	Let $N$ be a Wach module over $\AF^+$.
	Then, for any $k \in \ZZ$, the following sequence is exact:
	\begin{equation}\label{eq:fil_nmodmu_ker}
		0 \longrightarrow \Fil^{k-1} N \xrightarrow{\hspace{1mm} \mu \hspace{1mm}} \Fil^k N \longrightarrow \Fil^k (N/\mu N) \longrightarrow 0.
	\end{equation}
	In particular, we have that $\kert(\Fil^k N[1/p] \twoheadrightarrow \Fil^k \Dcrys(V)) = \mu \Fil^{k-1}N[1/p]$.
	Moreover, by taking the associated graded pieces, we obtain natural isomorphisms $(\gr^k N)/\mu (\gr^{k-1} N) \isomorphic \gr^k (N/\mu N)$ and $(\gr^k N[1/p])/\mu (\gr^{k-1} N[1/p]) \isomorphic \gr^k \Dcrys(V)$.
\end{lem}
\begin{proof}
	Exactness of \eqref{eq:fil_nmodmu_ker} easily follows from Lemma \ref{lem:wachmod_af+_twist_fil} (2).
	Then, by taking the associated graded pieces, we obtain the following exact sequence:
	\begin{equation*}
		0 \longrightarrow \gr^{k-1} N \xrightarrow{\hspace{1mm} \mu \hspace{1mm}} \gr^k N \longrightarrow \gr^k (N/\mu N) \longrightarrow 0.
	\end{equation*}
	Rest of the claims easily follow from these observations.
\end{proof}

\begin{rem}\label{rem:gamma_minus1_image}
	The Nygaard filtration on a Wach module $N$ over $\AF^+$ is stable under the action of $\Gamma_F$.
	Therefore, for any $g$ in $\Gamma_F$ and $k \in \ZZ$, using Lemma \ref{lem:wachmod_af+_twist_fil} (2) we see that $(g-1)\Fil^k N \subset (\Fil^k N) \cap \mu N = \mu \Fil^{k-1} N$.
\end{rem}

Finally, let us check the compatibility of the Nygaard filtration with exact sequences of Wach modules over $\AF^+$.
So consider the following $\AF^+\linear$ and $(\varphi, \Gamma_F)\equivariant$ exact sequence of Wach modules over $\AF^+$:
\begin{equation}\label{eq:wachmod_es}
	0 \longrightarrow N_1 \longrightarrow N_2 \longrightarrow N_3 \longrightarrow 0.
\end{equation}

\begin{lem}\label{lem:fil_left_exact}
	For $k \in \ZZ$, we have that $N_1 \cap \Fil^k N_2 = \Fil^k N_1$.
\end{lem}
\begin{proof}
	Let $D_i := \AF \otimes_{\AF^+} N_i$, for $i = 1, 2$.
	Note that we have $N_1 := D_1 \cap N_2 \subset D_2$.
	So, if $x$ is in $N_1 \cap \Fil^k N_2$, then we see that $\varphi(x)$ is in $D_1 \cap [p]_q^k N_2$, i.e.\ $[p]_q^{-k}\varphi(x)$ is in $D_1 \cap N_2 = N_1$.
	Hence, we get that $x$ in $\Fil^k N_1$.
\end{proof}

\begin{rem}\label{rem:mu_fil_intersect}
	For any $j, k \in \ZZ$, we have that $N_1 \cap \mu^j \Fil^k N_2 = \mu^j \Fil^k N_1$.
	Indeed, using the same notation as in the proof of Lemma \ref{lem:fil_left_exact}, we note that if $x$ is in $N_1 \cap \mu^j N_2$, then we may write $x = \mu^j y$ for some $y$ in $N_2$, and we see that $y = \mu^{-j} x$ is in $D_1 \cap N_2 = N_1$, i.e.\ $x$ is in $\mu^j N_1$.
	Combining this with Lemma \ref{lem:fil_left_exact} we get the claim.
\end{rem}

The statement of Lemma \ref{lem:fil_left_exact} can be strengthened after inverting $p$.
More precisely, we have the following:
\begin{prop}\label{prop:wachmod_fil_es}
	The following sequence is exact for each $k \in \ZZ$:
	\begin{equation}\label{eq:wachmod_fil_es}
		0 \longrightarrow \Fil^k N_1[1/p] \longrightarrow \Fil^k N_2[1/p] \longrightarrow \Fil^k N_3[1/p] \longrightarrow 0.
	\end{equation}
\end{prop}
\begin{proof}
	For each $i = 1, 2, 3$ and $r \in \ZZ$, from Remark \ref{rem:wachmod_cycltwist} note that $\mu^{-r}N_i(r)$, where $(r)$ denotes a twist by $\chi^r$, is again a Wach module over $\AF^+$, and \eqref{eq:wachmod_es} is exact if and only if the following sequence is exact:
	\begin{equation*}
		0 \longrightarrow \mu^{-r}N_1(r) \longrightarrow \mu^{-r}N_2(r) \longrightarrow \mu^{-r}N_3(r) \longrightarrow 0.
	\end{equation*}
	Now, let $P_i := N_i[1/p]$ for $i = 1, 2, 3$, and by using Lemma \ref{lem:wachmod_af+_twist_fil} (1), note that we have $\Fil^{k-r}(\mu^{-r} P_i(r)) = \mu^{-r} \Fil^k P_i (r)$.
	Therefore, we see that \eqref{eq:wachmod_fil_es} is exact if and only if the following is exact:
	\begin{equation*}
		0 \longrightarrow \Fil^{k-r} (\mu^{-r} P_1(r)) \longrightarrow \Fil^{k-r} (\mu^{-r} P_2(r)) \longrightarrow \Fil^{k-r} (\mu^{-r} P_2(r)) \longrightarrow 0.
	\end{equation*}
	In particular, without loss of generality we may assume that each $N_i$ is an effective Wach module over $\AF^+$, for $i \in \{1, 2, 3\}$, in particular, $\Fil^0 N_i = N_i$ and $\Fil^0 P_i = P_i$.
	We shall prove the claim by induction on $k \in \NN$.
	So let us assume the claim for $k-1$, and consider the following diagram:
	\begin{equation}\label{eq:wachmod_fil_diag}
		\begin{tikzcd}
			& 0 \arrow[d] & 0 \arrow[d] & 0 \arrow[d]\\
			0 \arrow[r] & \mu \Fil^{k-1} P_1 \arrow[r] \arrow[d] & \Fil^k P_1 \arrow[r] \arrow[d] & \Fil^k(P_1/\mu P_1) \arrow[r] \arrow[d] & 0\\
			0 \arrow[r] & \mu \Fil^{k-1} P_2 \arrow[r] \arrow[d] & \Fil^k P_2 \arrow[r] \arrow[d] & \Fil^k(P_2/\mu P_2) \arrow[r] \arrow[d] & 0\\
			0 \arrow[r] & \mu \Fil^{k-1} P_3 \arrow[r] \arrow[d] & (\Fil^k P_2)/(\Fil^k P_1) \arrow[r] \arrow[d] & \Fil^k(P_3/\mu P_3) \arrow[r] \arrow[d] & 0\\
			& 0 & 0 & 0.
		\end{tikzcd}
	\end{equation}
	In \eqref{eq:wachmod_fil_diag}, note that the top and the middle rows are exact by \eqref{eq:fil_nmodmu_ker} in Lemma \ref{lem:fil_nmodmu_ker}, and the left column is exact by the induction assumption.
	In the middle column, using that $P_1 = (\AF \otimes_{\AF^+} P_1) \cap P_2 \subset \AF \otimes_{\AF^+} P_2$, it easily follows that $\Fil^k P_1 \subset \Fil^k P_2$.
	Now, let $V_i := \VF(P_i)$, for each $i = 1, 2, 3$ (see Remark \ref{rem:crystalline_wach_equivalence_rational}).
	Then, from Theorem \ref{thm:qdeformation_dcrys}, we have filtered isomorphisms $\Fil^k(P_i/\mu P_i) \isomorphic \Fil^k \Dcrys(V_i)$.
	Recall that $\Dcrys$ is an exact functor, and in the category $\MF_F\wa(\varphi)$ (see Section \ref{subsec:padicreps}) exact sequences are compatible with filtration.
	So we get that the right column in \eqref{eq:wachmod_fil_diag} is also exact.
	Hence, it follows that the bottom row in \eqref{eq:wachmod_fil_diag} is exact.
	Now, consider the following commutative diagram with exact rows:
	\begin{center}
		\begin{tikzcd}
			0 \arrow[r] & \Fil^k P_1 \arrow[r] \arrow[d] & \Fil^k P_2 \arrow[r] \arrow[d] & (\Fil^k P_2)/(\Fil^k P_1) \arrow[r] \arrow[d] & 0\\
			0 \arrow[r] & P_1 \arrow[r] & P_2 \arrow[r] & P_3 \arrow[r] & 0,
		\end{tikzcd}
	\end{center}
	where the left and middle vertical arrows are natural injective maps, and the right vertical arrow is injective because $\Fil^k P_1 = P_1 \cap \Fil^k P_2 \subset P_2$ (see Lemma \ref{lem:fil_left_exact}).
	Moreover, chasing an element of $(\Fil^k P_2)/(\Fil^k P_1)$ easily shows that the right vertical arrow factors as injective maps $(\Fil^k P_2)/(\Fil^k P_1) \hookrightarrow \Fil^k P_3 \hookrightarrow P_3$.
	In particular, the discussion above induces a natural map from the bottom exact row of \eqref{eq:wachmod_fil_diag} to the exact sequence \eqref{eq:fil_nmodmu_ker} for $P_3$ (see Lemma \ref{lem:fil_nmodmu_ker}):
	\begin{center}
		\begin{tikzcd}
			0 \arrow[r] & \mu \Fil^{k-1} P_3 \arrow[r] \arrow[d, equal] & (\Fil^k P_2)/(\Fil^k P_1) \arrow[r] \arrow[d] & \Fil^k(P_3/\mu P_3) \arrow[r] \arrow[d, equal] & 0\\
			0 \arrow[r] & \mu \Fil^{k-1} P_3 \arrow[r] & \Fil^k P_3 \arrow[r] & \Fil^k(P_3/\mu P_3) \arrow[r] & 0,
		\end{tikzcd}
	\end{center}
	where the middle vertical arrow is injective.
	By five lemma, it follows that $(\Fil^k P_2)/(\Fil^k P_1) \isomorphic \Fil^k P_3$, thus proving the claim.
\end{proof}

\begin{rem}
	The statement of Proposition \ref{prop:wachmod_fil_es} is of independent interest, and should be seen as a statement parallel to the fact that an exact sequence in the category $\MF_F\wa(\varphi)$ is strict, i.e.\ it remains exact on filtrations.
	Moreover, Proposition \ref{prop:wachmod_fil_es} also ensures that the functor from the category of Wach modules over $\AF^+[1/p]$ to the category of complexes of $\ZZ_p\modules$ given as $N \mapsto \pazs^{\bullet}(N)$ (see Definition \ref{defi:syntomic_complex_af+}), i.e.\ by sending $N$ to the syntomic complex with coefficients in $N$, is an exact functor.
\end{rem}

\subsubsection{Filtration on Wach modules over \texorpdfstring{$S$}{-}}\label{subsubsec:nyfil_m}

Let $M$ be a Wach module over $S$.
We equip $M$ with a decreasing filtration called the \textit{Nygaard filtration} as
\begin{equation}\label{eq:nyfil_m}
	\Fil^k M := \{ x \in N \textrm{ such that } \varphi(x) \in \ptilde^k M\}, \textrm{ for } k \in \ZZ.
\end{equation}
From the definition it is clear that $M$ is effective if and only if $\Fil^0 M = M$.
Now, let $N := \AF^+ \otimes_S M$, and note that the natural $S\linear$ map $M \rightarrow N$ is injective and $(\varphi, \Gamma_F)\equivariant$.
Moreover, it is easy to see that $\Fil^k M = M \cap \Fil^k N \subset N$, where the Nygaard filtration on $N$ was defined in \eqref{eq:nyfil_n}.
Similar to Lemma \ref{lem:wachmod_af+_twist_fil}, we claim the following:

\begin{lem}\label{lem:wachmod_s_twist_fil}
	For all $k \in \ZZ$, we have that $\Fil^k M \cap \mu_0 M = \mu_0\Fil^{k-p+1} M \subset M[1/\mu_0]$.
\end{lem}
\begin{proof}
	From Lemma \ref{lem:wachmod_af+_twist_fil}, note that we have $\Fil^k N \cap \mu N = \mu \Fil^{k-1} N$.
	Moreover, recall that $\mu_0$ is given as the product of $\mu^{p-1}$ with a unit in $\AF^+$.
	Therefore, it follows that  $\Fil^k N \cap \mu_0 N = \Fil^k N \cap \mu^{p-1} N = \mu^{p-1} \Fil^{k-p+1} N = \mu_0 \Fil^{k-p+1} N$.
	Hence, we obtain that $\Fil^k M \cap \mu_0 M = \Fil^k N \cap \mu_0 M = \mu_0 \Fil^{k-p+1} N \cap \mu_0 M = \mu_0 \Fil^{k-p+1} M$.
\end{proof}

Next, we equip the $O_F\module$ $M/\mu_0 M$ with a filtration $\Fil^k(M/\mu_0 M)$, for $k \in \ZZ$, and given as the image of $\Fil^k M$ under the surjection $M \twoheadrightarrow M/\mu_0 M$.
Moreover, we equip the $O_F\module$ $N/\mu N$ with a filtration $\Fil^k(N/\mu N)$, for $k \in \ZZ$, and given as the image of $\Fil^k N$ under the surjection $N \twoheadrightarrow N/\mu N$ (see Section \ref{subsubsec:nyfil_wachmod_af+}).
From Proposition \ref{prop:wachmod_fpx_descent}, we have that the natural $S\linear$ and $(\varphi, \Gamma_F)\equivariant$ map $M \rightarrow N$ induces an $O_F\linear$ isomorphism $M/\mu_0 M \isomorphic N/\mu N$ compatible with the respective Frobenii.
Then, we claim the following:
\begin{lem}\label{lem:mmodmu0_nmodmu_fil}
	The natural isomorphism $M/\mu_0 M \isomorphic N/\mu N$ induces an $O_F\linear$ isomorphism $\Fil^k (M/\mu_0 M) \isomorphic \Fil^k (N/\mu N)$, for each $k \in \ZZ$.
\end{lem}
\begin{proof}
	Consider the following $(\varphi, \Gamma_F)\equivariant$ commutative diagram:
	\begin{equation}\label{eq:mmodmu0_nmodmu_fil}
		\begin{tikzcd}
			0 \arrow[r] & \mu_0 \Fil^{k-p+1} M \arrow[r] \arrow[d] & \Fil^k M \arrow[r] \arrow[d] & \Fil^k(M/\mu_0 M) \arrow[r] \arrow[d] & 0\\
			0 \arrow[r] & \mu \Fil^{k-1} N \arrow[r] & \Fil^k N \arrow[r] & \Fil^k(N/\mu N) \arrow[r] & 0,
		\end{tikzcd}
	\end{equation}
	where the top row is exact by Lemma \ref{lem:wachmod_s_twist_fil}, and the bottom row is the exact sequence \eqref{eq:fil_nmodmu_ker} in Lemma \ref{lem:fil_nmodmu_ker}.
	Now, it is clear that $\mu_0 \Fil^{k-p+1} M \subset \mu \Fil^{k-1} N \cap \Fil^k M$ (since $\mu_0$ is the product of $\mu^{p-1}$ with a unit in $\AF^+$, see Section \ref{subsec:period_rings}).
	Conversely, note that we have $\mu \Fil^{k-1} N \cap \Fil^k M \subset \mu N \cap M \cap \Fil^k M = \mu_0 M \cap \Fil^k M = \mu_0 \Fil^{k-p+1} M$.
	Therefore, it follows that the right vertical arrow in \eqref{eq:mmodmu0_nmodmu_fil} is injective, and we claim that it is surjective as well.
	Indeed, let $x$ be an element of $\Fil^k(N/\mu N)$, and let $y$ in $M$ be a lift of $x$, under the composition $M \twoheadrightarrow M/\mu_0 M \isomorphic N/\mu N$.
	Then, via the natural $S\linear$ and $(\varphi, \Gamma_F)\equivariant$ injective homomorphism $M \rightarrow N$, we see that $y$ is in $N$ and a lift of $x$.
	In particular, we get that $y$ is in $\Fil^k N \cap M = \Fil^k M$.
	Taking the image of $y$ under the map $\Fil^k M \rightarrow \Fil^k(M/\mu_0 M)$ gives a lifting of $x$ under the right vertical map of \eqref{eq:mmodmu0_nmodmu_fil}.
	Hence, we obtain that $\Fil^k (M/\mu_0 M) \isomorphic \Fil^k (N/\mu N)$.
\end{proof}

Now, note that $(M/\mu_0 M)[1/p]$ is a $\varphi\module$ over $F$ since $\ptilde = p \textmod \mu_0 M$, and $M/\mu_0 M$ is equipped with a filtration $\Fil^k(M/\mu_0 M)$ as above.
We equip $(M/\mu_0 M)[1/p]$ with the induced filtration $\Fil^k((M/\mu_0 M)[1/p]) := \Fil^k(M/\mu_0 M)[1/p]$, and note that it is a filtered $\varphi\module$ over $F$.
By combining Theorem \ref{thm:qdeformation_dcrys} and Lemma \ref{lem:mmodmu0_nmodmu_fil}, we get the following:
\begin{thm}\label{thm:mu0_deformation_dcrys}
	Let $M$ be a Wach module over $S$ and $V := \TF(M)[1/p]$ the associated crystalline representation of $G_F$ from Theorem \ref{thm:crystalline_wach_s_equivalence}.
	Then, we have that $(M/\mu_0 M)[1/p] \isomorphic \Dcrys(V)$ as filtered $\varphi\modules$ over $F$.
\end{thm}

From Theorem \ref{thm:mu0_deformation_dcrys} we have a surjection $\Fil^k M[1/p] \twoheadrightarrow \Fil^k \Dcrys(V)$, and we would like to determine its kernel.
\begin{lem}\label{lem:fil_mmodmu0_ker}
	Let $M$ be a Wach module over $S$.
	Then, for any $k \in \ZZ$, the following sequence is exact:
	\begin{equation}\label{eq:fil_mmodmu0_ker}
		0 \longrightarrow \Fil^{k-p+1} M \xrightarrow{\hspace{1mm} \mu_0 \hspace{1mm}} \Fil^k M \longrightarrow \Fil^k (M/\mu_0 M) \longrightarrow 0.
	\end{equation}
	In particular, we have that $\kert(\Fil^k M[1/p] \twoheadrightarrow \Fil^k \Dcrys(V)) = \mu_0 \Fil^{k-p+1}M[1/p]$.
	Moreover, by taking the associated graded pieces, we obtain natural isomorphisms $(\gr^k M)/\mu_0(\gr^{k-p+1} M) \isomorphic \gr^k (M/\mu_0 M)$ and $(\gr^k M[1/p])/\mu_0(\gr^{k-p+1} M[1/p]) \isomorphic \gr^k \Dcrys(V)$.
\end{lem}
\begin{proof}
	Exactness of \eqref{eq:fil_mmodmu0_ker} easily follows from Lemma \ref{lem:wachmod_s_twist_fil}.
	Then, by taking the associated graded pieces, we obtain the following exact sequence:
	\begin{equation*}
		0 \longrightarrow \gr^{k-p+1} M \xrightarrow{\hspace{1mm} \mu_0 \hspace{1mm}} \gr^k M \longrightarrow \gr^k (M/\mu M) \longrightarrow 0.
	\end{equation*}
	Rest of the claims easily follow from these observations.
\end{proof}

\begin{rem}\label{rem:gamma0_minus1_image}
	The Nygaard filtration on a Wach module $M$ over $S$ is stable under the action of $\Gamma_0$.
	Therefore, for any $g$ in $\Gamma_0$ and $k \in \ZZ$, using Lemma \ref{lem:wachmod_s_twist_fil} we see that $(g-1)\Fil^k M \subset (\Fil^k M) \cap \mu_0 M = \mu_0 \Fil^{k-p+1} M$.
\end{rem}

\section{Syntomic complexes and Galois cohomology}\label{sec:syntomic_complex_galcoh}

In this section, we will define syntomic complexes with coefficients in a Wach module over $\AF^+$ and $S$, respectively, and show that -- after inverting $p$ -- our complexes compute the crystalline part of the Galois cohomology of the associated crystalline representation (see Theorem \ref{thm:syntomic_blochkatoselmer_af+} and Theorem \ref{thm:syntomic_blochkatoselmer_s}).

\subsection{Syntomic complex over \texorpdfstring{$\AF^+$}{-}}\label{subsec:syntomic_complex_af+}

Let $N$ be a Wach module over $\AF^+$, and define an operator $\nabla_q := \tfrac{\gamma-1}{\mu} \colon N \rightarrow N$.
From Remark \ref{rem:gamma_minus1_image}, note that we have $\nabla_q(\Fil^k N) \subset \Fil^{k-1} N$, for each $k \in \ZZ$.
\begin{defi}\label{defi:syntomic_complex_af+}
	Define the \textit{syntomic complex} with coefficients in $N$ as
	\begin{equation}\label{eq:syntomic_complex_af+}
		\pazs^{\bullet}(N) \colon \Fil^0 N \xrightarrow{\hspace{1mm}(\nabla_q, 1-\varphi)\hspace{1mm}} \Fil^{-1} N \oplus N \xrightarrow{\hspace{1mm}(1-[p]_q\varphi, \nabla_q)^{\intercal}\hspace{1mm}} N,
	\end{equation}
	where the first map is given as $x \mapsto (\nabla_q(x), (1-\varphi)x)$, and the second map is given as $(x, y) \mapsto (1-[p]_q\varphi)x - \nabla_q(y)$.
\end{defi}

The goal of this section is to show the following claim:
\begin{thm}\label{thm:syntomic_blochkatoselmer_af+}
	Let $N$ be a Wach module over $\AF^+$ and $V = \TF(N)[1/p]$ the associated $\padic$ crystalline representation of $G_F$ from Theorem \ref{thm:crystalline_wach_af+_equivalence}.
	Then, we have a natural isomorphism, for each $k \in \NN$,
	\begin{equation*}
		H^k(\pazs^{\bullet}(N))[1/p] \isomorphic H^k_f(G_F, V).
	\end{equation*}
\end{thm}
\begin{proof}
	The claim for $H^0_f$ follows from Lemma \ref{lem:h0_syn_blochkatoselmer}.
	For $H^1_f$ recall that from Remark \ref{rem:h1f_extclass} we have a natural isomorphism 
	\begin{equation*}
		H^1_f(G_F, V) \isomorphic \Ext^1_{\Rep_{\QQ_p}^{\crys}(G_F)}(\QQ_p, V).
	\end{equation*}
	Moreover, from Remark \ref{rem:crystalline_wach_equivalence_rational} the functors $\NF$ and its quasi-inverse $\VF$ are exact.
	Therefore, we have a natural isomorphism
	\begin{equation*}
		\Ext^1_{(\varphi, \Gamma_F)\Mod_{\AF^+[1/p]}^{[p]_q}}(\AF^+[1/p], N[1/p]) \isomorphic \Ext^1_{\Rep_{\QQ_p}^{\crys}(G_F)}(\QQ_p, V).
	\end{equation*}
	Combining these observations with Corollary \ref{cor:h1syn_ext1wach_bf+}, we get a natural isomorphism 
	\begin{equation*}
		H^1(\pazs^{\bullet}(N))[1/p] \isomorphic H^1_f(G_F, V).
	\end{equation*}
	Finally, note that the Wach module $N$ over $\AF^+$ can always be written as a twist of an effective Wach module over $\AF^+$, and similarly, the representation $V = \TF(N)[1/p]$ is the twist of the corresponding positive crystalline representation (i.e.\ all Hodge--Tate weights less than or equal to 0) by a power of the cyclotomic character (see Remark \ref{rem:wachmod_cycltwist}).
	Therefore, the claim for $H^2_f$ follows from Proposition \ref{prop:h2_syn_blochkatoselmer}, thus completing our proof.
\end{proof}

\begin{rem}\label{rem:syntomic_to_blochkato_af+}
	Let us consider the following diagram of complexes:
	\begin{equation}\label{eq:syntomic_to_blochkato_af+}
		\begin{tikzcd}[column sep=large]
			\mu \Fil^{-1} N \arrow[r, "(\gamma-1 {,} 1-\varphi)"] \arrow[d] & \mu\Fil^{-1} N \oplus \mu N \arrow[r, "(1-\varphi {,} \gamma-1)^{\intercal}"] \arrow[d] &[5mm] \mu N \arrow[d, equal]\\
			\Fil^0 N \arrow[r, "(\gamma-1 {,} 1-\varphi)"] \arrow[d] & \mu\Fil^{-1} N \oplus N \arrow[r, "(1-\varphi {,} \gamma-1)^{\intercal}"] \arrow[d] & \mu N \\
			\Fil^0 (N/\mu N) \arrow[r, "1-\varphi"] & N/\mu N,
		\end{tikzcd}
	\end{equation}
	where the vertical arrows are natural maps, and the complex in the middle row is isomorphic to the complex $\pazs^{\bullet}(N)$ in \eqref{eq:syntomic_complex_af+}.
	Let $V := \TF(N)[1/p]$ from Theorem \ref{thm:crystalline_wach_af+_equivalence}.
	Then, after inverting $p$ and using Theorem \ref{thm:qdeformation_dcrys}, we see that the complex in the bottom row of \eqref{eq:syntomic_to_blochkato_af+} is the same as the complex $\cald^{\bullet}(\Dcrys(V))$ in \eqref{eq:dcrys_complex_bks}.
	Moreover, note that in \eqref{eq:syntomic_to_blochkato_af+}, the middle column is exact by Theorem \ref{thm:qdeformation_dcrys}, and the left-hand side column is exact by \eqref{eq:fil_nmodmu_ker} in Lemma \ref{lem:fil_nmodmu_ker}.
	Therefore, by Theorem \ref{thm:syntomic_blochkatoselmer_af+} and Corollary \ref{cor:dcrys_complex_bks}, it follows that the diagram \eqref{eq:syntomic_to_blochkato_af+} induces a natural quasi-isomorphism of complexes $\pazs^{\bullet}(N)[1/p] \simeq \cald^{\bullet}(\Dcrys(V))$.
\end{rem}

\begin{rem}\label{rem:integral_bks}
	In \cite[Equation (3.7.3)]{bloch-kato}, Bloch and Kato also defined an integral version of the Bloch--Kato Selmer group.
	Namely, if $T$ is a finite free $\ZZ_p\textrm{-representation}$ of $G_F$, then using the natural map $\iota \colon H^1(G_F, T) \rightarrow H^1(G_F, T[1/p])$, they set $H^i_f(G_F, T) := \iota^{-1}(H^i_f(G_F, T[1/p]))$.
	Now, assume that $T[1/p]$ is crystalline, and let $N := \NF(T)$ denote the Wach module over $\AF^+$ associated to $T$.
	Then, in light of Theorem \ref{thm:syntomic_blochkatoselmer_af+}, it is natural to ask whether $H^i(\pazs^{\bullet}(N))$ computes $H^i_f(G_F, T)$.

	The answer to the preceding question does not appear to be positive for $i=1$ because the Wach module functor does not preserve short exact sequences (see Remark \ref{rem:crystalline_wach_equivalence_rational}).
	One could possibly remedy this shortcoming by enlarging the category of Wach modules to include ``torsion'' Wach modules.
	In this direction, we expect that ideas appearing in \cite{bhatt-lurie-fgauges} could provide useful insights in gaining a better understanding of integral structures.
	On the other hand, note that the quasi-inverse functor $T_F$ from Theorem \ref{thm:crystalline_wach_af+_equivalence} preserves short exact sequences.
	Indeed, the functor $T_F$ naturally factors through $(\varphi, \Gamma_F)\Mod_{\AF^+}^{[p]_q} \rightarrow (\varphi, \Gamma_F)\Mod_{\AF}^{\etale} \isomorphic \Rep_{\ZZ_p}^{\crys}(G_F)$, where the left functor is exact (see Remark \ref{rem:wachtoetale_af+}), and the right equivalence is also exact (see Section \ref{subsec:padicreps}).
	So it is reasonable to expect that $H^i(\pazs^{\bullet}(N))$ computes a subspace of $H^i_f(G_F, T)$.
\end{rem}

In the rest of this section, we will compute the cohomology of the complex $\pazs^{\bullet}(N)$ from \eqref{eq:syntomic_complex_af+}.

\subsubsection{Comparing \texorpdfstring{$H^0$}{-} and \texorpdfstring{$H^1$}{-}}\label{subsubsec:integral_comp_h0h1}

In this section, we shall compute $H^0$ and $H^1$ of the complex $\pazs^{\bullet}(N)$. 
\begin{lem}\label{lem:h0_syn_blochkatoselmer}
	Let $N$ be a Wach module over $\AF^+$ and $T = \TF(N)$ the associated $\ZZ_p\representation$ of $G_F$ from Theorem \ref{thm:crystalline_wach_af+_equivalence} such that $T[1/p]$ is crystalline.
	Then, we have a natural isomorphism
	\begin{equation*}
		H^0(\pazs^{\bullet}(N)) = (\Fil^0 N)^{\varphi=1, \nabla_q=0} \isomorphic T^{G_F}.
	\end{equation*}
\end{lem}
\begin{proof}
	Note that a simple computation shows that we have $(\Fil^0 N)^{\varphi=1, \nabla_q=0} = (\Fil^0 N)^{\varphi=1, \gamma=1} = N^{\varphi=1, \gamma=1}$.
	Now, let $M := N^{\FF_p^{\times}}$, and recall that we have an $\AF^+\linear$ and $(\varphi, \Gamma_F)\equivariant$ isomorphism $\AF^+ \otimes_S M \isomorphic N$ (see Proposition \ref{prop:wachmod_fpx_descent}).
	Therefore, we see that $N^{\varphi=1, \gamma=1} = (\AF^+ \otimes_S M)^{\varphi=1, \gamma=1} = M^{\varphi = 1, \gamma=1} = M^{\varphi = 1, \Gamma_0} = N^{\varphi = 1, \Gamma_F}$, where the third equality follows by the continuity of the action of $\Gamma_0$ on $M$.
	Moreover, note that $(\AF^+[1/\mu])^{\gamma=1} = O_F$, therefore, a similar argument shows that we have $(N[1/\mu])^{\varphi=1, \gamma=1} = (N[1/\mu])^{\varphi = 1, \Gamma_F}$.
	We claim that $N^{\varphi=1, \gamma=1} = (N[1/\mu])^{\varphi=1, \gamma=1}$.
	Indeed, let $(x/\mu^k)$ be in $N[1/\mu]^{\varphi=1, \gamma=1}$, for some $x$ in $N$ and $k \in \ZZ$.
	Then, it is enough to show that $x$ is in $\mu^k N$.
	Note that $\gamma$ is a topological generator of $\Gamma_0$, and we have $\gamma(x) = (\gamma(\mu)^k/\mu^k)x$.
	So, reduction modulo $\mu$ gives $\gamma(x) = \chi(\gamma)^k x \textrm{ mod } \mu N$.
	Since $\Gamma_F$ acts trivially on $N/\mu N$ and $\chi(\gamma)^k-1$ is a unit in $\AF^+[1/p]$, we obtain that $x$ is in $\mu N[1/p] \cap N = \mu N$.
	Iterating this $k$ times, we obtain that $x$ is in $\mu^k N$, as claimed.
	In particular, we have that $N^{\varphi = 1, \Gamma_F} = (N[1/\mu])^{\varphi = 1, \Gamma_F}$, and it remains to compute the latter term.

	From Proposition \ref{prop:wachmod_rep_comp}, recall that for a certain ring $A^+$ equipped with a natural action of $(\varphi, G_F)$, we have an $A^+[1/\mu]\linear$ and $(\varphi, G_F)\equivariant$ comparison isomorphism $A^+[1/\mu] \otimes_{\AF^+} N \isomorphic A^+[1/\mu] \otimes_{\ZZ_p} T$.
	Moreover, we have that $(A^+)^{H_F} = \AF^+$ and $(A^+[1/\mu])^{\varphi=1} = \ZZ_p$, and the action of $\varphi$ and $G_F$ commute with each other, therefore, by taking the fixed points of the preceding isomorphism under the action of $\varphi$ and $G_F$, yields
	\begin{equation*}
		(N[1/\mu])^{\varphi=1, \Gamma_F} = (A^+[1/\mu] \otimes_{\AF^+} N)^{\varphi = 1, G_F} \isomorphic (A^+[1/\mu] \otimes_{\ZZ_p} T)^{\varphi = 1, G_F} = T^{G_F}.
	\end{equation*}
	This allows us to conclude.
\end{proof}

\begin{nota}
	Let $N$ and $N'$ be two Wach modules over $\AF^+$.
	We say that a Wach module $E$ over $\AF^+$ is an extension of $N$ by $N'$ if it fits into a $(\varphi, \Gamma_F)\equivariant$ exact sequence of $\AF^+\textrm{-modules}$, $0 \rightarrow N' \rightarrow E_1 \rightarrow N \rightarrow 0$.
	Two extensions $E_1$ and $E_2$ are equivalent if there exists a $(\varphi, \Gamma_F)\equivariant$ commutative diagram of $\AF^+\textrm{-modules}$ with exact rows:
	\begin{center}
		\begin{tikzcd}
			0 \arrow[r] & N' \arrow[r] \arrow[d, equal] & E_1 \arrow[r] \arrow[d, "\wr"] & N \arrow[r] \arrow[d, equal] & 0\\
			0 \arrow[r] & N' \arrow[r] & E_2 \arrow[r] & N \arrow[r] & 0.
		\end{tikzcd}
	\end{center}
	Let $\Ext^1_{(\varphi, \Gamma_F)\Mod_{\AF^+}^{[p]_q}}(N, N')$ denote the set of equivalence classes of extensions of $N$ by $N'$.
\end{nota}

\begin{prop}\label{prop:h1syn_ext1wach_af+}
	Let $N$ be a Wach module over $\AF^+$.
	Then, we have a natural bijection of sets
	\begin{equation}\label{eq:h1syn_ext1wach_af+}
		H^1(\pazs^{\bullet}(N)) \isomorphic \Ext^1_{(\varphi, \Gamma_F)\Mod_{\AF^+}^{[p]_q}}(\AF^+, N).
	\end{equation}
\end{prop}
\begin{proof}
	To prove the claim, we shall construct a map
	\begin{equation*}
		\alpha \colon H^1(\pazs^{\bullet}(N)) \longrightarrow \Ext^1_{(\varphi, \Gamma_F)\Mod_{\AF^+}^{[p]_q}}(\AF^+, N),
	\end{equation*}
	and show that it is bijective by constructing an inverse map.
	Let $(x, y)$ represent a class in $H^1(\pazs^{\bullet}(N))$, i.e.\ we have $x$ in $\Fil^{-1} N$ and $y$ in $N$ such that $(1-[p]_q\varphi)x = \nabla_q(y)$.
	Set $E_1 := N \oplus \AF^+ \cdot e$ with $\gamma(e) = \mu x + e$, $\varphi(e) = y + e$ and $g(e) = e$, for $g$ a generator of $\Gamma_{\textrm{tor}} \isomorphic \FF_p^{\times}$.
	Clearly, $E_1$ is a Wach module over $\AF^+$.
	Moreover, by sending $e$ to the identity element in $\AF^+$, we obtain an exact sequence of Wach modules over $\AF^+$,
	\begin{equation*}
		0 \longrightarrow N \longrightarrow E_1 \longrightarrow \AF^+ \longrightarrow 0,
	\end{equation*}
	This represents an extension class of $\AF^+$ by $N$, and we set $\alpha[(x, y)] = [E_1]$, where we represent cohomological (resp.\ extension) classes with ``$[\hspace{0.5mm}]$''.
	To show that $\alpha$ is well-defined we need to show that the extension class $[E_1]$ is independent of the choice of the presentation $(x, y)$.
	Indeed, let $(x', y')$ be another presentation such that $x'-x = \nabla_q(w)$, $y'-y = (1-\varphi)w$ for some $w$ in $\Fil^0 N$.
	Then, similar to above note that $E_2 :=  N \oplus \AF^+ \cdot e'$, with $\gamma(e') = \mu x' + e'$, $g(e') = e'$ and $\varphi(e') = y' + e'$, is a Wach module over $\AF^+$ and an extension of $\AF^+$ by $N$.
	Let us define an $\AF^+\linear$ map $f \colon E_1 \rightarrow E_2$ given as identity on $N$, and we set $f(e) = e'-w$.
	Then, $f$ is bijective because we may define $f^{-1} \colon E_2 \rightarrow E_1$, given as the identity on $N$, and by setting $f^{-1}(e') = e+w$ it is easy to verify that $f \circ f^{-1} = id$ and $f^{-1} \circ f = id$.
	From the formulas $x'-x = \nabla_q(w)$ and $y'-y = (1-\varphi)y$ it follows that $f$ and $f^{-1}$ are $(\varphi, \Gamma_F)\equivariant$.
	Now, consider the following diagram with $\AF^+\linear$ maps and exact rows:
	\begin{center}
		\begin{tikzcd}
			0 \arrow[r] & N \arrow[r] \arrow[d, equal] & E_1 \arrow[r] \arrow[d, "f"', "\wr"] & \AF^+ \arrow[r] \arrow[d, equal] & 0\\
			0 \arrow[r] & N \arrow[r] & E_2 \arrow[r] & \AF^+ \arrow[r] & 0.
		\end{tikzcd}
	\end{center}
	The left square commutes by the definition of $f$.
	Moreover, the $\AF^+\linear$ map $E_1 \rightarrow \AF^+$ sends $e \mapsto 1$, and the $\AF^+\linear$ map $E_2 \rightarrow \AF^+$ sends $e' \mapsto 1$, therefore, it follows that the right square commutes as well.
	Hence, $E_1$ and $E_2$ represent the same extension class of $\AF^+$ by $N$, in the category $(\varphi, \Gamma_F)\Mod_{\AF^+}^{[p]_q}$.
	In particular, $\alpha$ is a well-defined map.

	Next, let us construct an inverse of $\alpha$ which we will denote by $\beta$.
	Consider an extension of Wach modules over $\AF^+$ as
	\begin{equation*}
		0 \longrightarrow N \longrightarrow E_1 \longrightarrow \AF^+ \longrightarrow 0.
	\end{equation*}
	We write $E_1 = N \oplus \AF^+ \cdot e$, where $e$ in $E_1$ is a lift of the identity element in $\AF^+$, and we have $(\gamma-1) e = z$ and $(1-\varphi) e = y$ for some $y$, $z$ in $N$.
	But, then we have that $\varphi(e) = e-y$ is in $E_1$, i.e.\ $e$ is in $\Fil^0 E_1$.
	Therefore, we get that $z = (\gamma-1)e$ is in $N \cap \mu \Fil^{-1} E_1 = \mu \Fil^{-1} N$, where the last equality follows from Remark \ref{rem:mu_fil_intersect}.
	In particular, we obtain that $\nabla_q(e) = \frac{\gamma-1}{\mu} e = x$, for some $x$ in $\Fil^{-1} N$.
	By the commutativity of the action of $\varphi$ and $\gamma$, we get that $(1-[p]_q\varphi) \circ \nabla_q(e) = \nabla_q \circ (1-\varphi)e$, or equivalently
	\begin{equation*}
		(1-[p]_q\varphi)x = \nabla_q(y).
	\end{equation*}
	Therefore, we see that $(x, y)$ represents a cohomological class in $H^1(\pazs^{\bullet}(N))$, and we set $\beta([E_1]) = [(x, y)]$.
	Let us first show that the class $[(x, y)]$ is independent of the lift $e$ in $E_1$ of the identity element in $\AF^+$.
	So, let $e'$ in $E$ denote another lift of the identity element in $\AF^+$.
	Then, similar to above we have that $e'$ is in $\Fil^0 E$, and there exist $x'$ in $\Fil^{-1} N$ and $y'$ in $N$ such that $\nabla_q(e') = x'$, $(1-\varphi)e' = y'$ and $(1-[p]_q\varphi)x' = \nabla_q(y')$.
	Moreover, from Lemma \ref{lem:fil_left_exact}, we note that $w = e' - e$ is in $\Fil^0 E \cap N = \Fil^0 N$, in particular, we get that $x' = x + \nabla_q(w)$ and $y' = y + (1-\varphi)w$.
	Since $(1-[p]_q\varphi) \circ \nabla_q = \nabla_q \circ (1-\varphi)$, therefore, we conclude that $(x, y)$ and $(x', y')$ represent the same class in $H^1(\pazs^{\bullet}(N))$.
	Now, to show that $\beta$ is well-defined, we must show that the class $[(x, y)]$ is independent of the presentation $E_1$ of the extension class $[E_1]$.
	So let $E_2$ denote another presentation of the extension class $[E_1]$, i.e.\ $E_2$ is a Wach module over $\AF^+$, and there exists a $(\varphi, \Gamma_F)\equivariant$ isomorphism $f \colon E_1 \isomorphic E_2$ fitting into the following $(\varphi, \Gamma_F)\equivariant$ commutative diagram of $\AF^+\textrm{-modules}$ with exact rows:
	\begin{center}
		\begin{tikzcd}
			0 \arrow[r] & N \arrow[r] \arrow[d, equal] & E_1 \arrow[r] \arrow[d, "f"', "\wr"] & \AF^+ \arrow[r] \arrow[d, equal] & 0\\
			0 \arrow[r] & N \arrow[r] & E_2 \arrow[r] & \AF^+ \arrow[r] & 0.
		\end{tikzcd}
	\end{center}
	Let $e''$ in $E_2$ denote a lift of the identity element in $\AF^+$, and arguing as above we have that $e''$ is in $\Fil^0 E_2$, and there exist some $x''$ in $\Fil^{-1} N$ and $y''$ in $N$ such that $\nabla_q(e'') = x''$, $(1-\varphi)e'' = y''$ and $(1-[p]_q\varphi)x'' = \nabla_q(y'')$.
	Then, from the commutative diagram above we have that $f^{-1}(e'')$ in $E_1$ denotes a lift of the identity element in $\AF^+$, therefore, it follows that $f^{-1}(e'')$ is in $\Fil^0 E_1$, and $\nabla_q(f^{-1}(e'')) = x''$, $(1-\varphi)f^{-1}(e'') = y''$ and $(1-[p]_q\varphi)f^{-1}(x'') = \nabla_q(f^{-1}(y''))$.
	Using that the extension class $[E_1]$ is independent of the choice of a lift in $E_1$ of the identity element in $\AF^+$, it follows that $(x, y)$ and $(x'', y'')$ represent the same cohomological class in $H^1(\pazs^{\bullet}(N))$.
	In particular, we obtain that $\beta$ is a well-defined map.
	
	Finally, it remains to show that the two constructions described above are inverse to each other, i.e.\ $\alpha \circ \beta = id$ and $\beta \circ \alpha = id$.
	Note that starting with a class $[(x, y)]$ in $H^1(\pazs^{\bullet}(N))$ we may construct an extension $E$ of $\AF^+$ by $N$ in $(\varphi, \Gamma_F)\Mod_{\AF^+}^{[p]_q}$, such that $[E] = \alpha[(x, y)]$, i.e.\ $E$ may be described using the pair $(x, y)$.
	After applying $\beta$ we obtain a class $\beta([E]) = [(x', y')]$ in $H^1(\pazs^{\bullet}(N))$ with a presentation $(x', y')$ depending on the choice of some lift in $E$ of the identity element in $\AF^+$.
	Note that by construction, $E$ admits two descriptions using $(x, y)$ and $(x', y')$, respectively, depending on the choice of the lift in $E$ of the identity element in $\AF^+$.
	As we have shown that the class $[E]$ is independent of this choice, therefore, it follows that $[(x, y)] = [(x', y')] = \beta \circ \alpha [(x, y)]$ in $H^1(\pazs^{\bullet}(N))$.
	Next, starting with an extension $E$ of $\AF^+$ by $N$ in $(\varphi, \Gamma_F)\Mod_{\AF^+}^{[p]_q}$, we may construct a class $[(x, y)] = \beta([E])$ in $H^1(\pazs^{\bullet}(N))$.
	After applying $\alpha$, we obtain an extension class $[E'] = \alpha[(x, y)]$, where $E'$ is an extension of $\AF^+$ by $N$ in $(\varphi, \Gamma_F)\Mod_{\AF^+}^{[p]_q}$.
	By construction, we may write $E = N \oplus \AF^+ \cdot e$, with $\nabla_q(e) = x$, $g(e) = e$ and $(1-\varphi)e = y$, and $E' = N \oplus \AF^+ \cdot e'$ with $\nabla_q(e') = x$, $g(e) = e$ and $(1-\varphi)e' = y$.
	Now, note that the $\AF^+\linear$ map $f \colon E \rightarrow E'$, defined using the identity on $N$, and by setting $f(e)= e'$, is a $(\varphi, \Gamma_F)\equivariant$ isomorphism, in particular, we have that $[E] = [E'] = \alpha \circ \beta ([E])$.
	Hence, we have shown that the map $\alpha$ is a natural bijection.
\end{proof}

\begin{cor}\label{cor:h1syn_ext1wach_bf+}
	Let $N$ be a Wach module over $\AF^+$, and set $\BF^+ := \AF^+[1/p]$.
	Then, we have a natural isomorphism of abelian groups:
	\begin{equation}\label{eq:h1syn_ext1wach_bf+}
		H^1(\pazs^{\bullet}(N))[1/p] = H^1(\pazs^{\bullet}(N)[1/p]) \isomorphic \Ext^1_{(\varphi, \Gamma_F)\Mod_{\BF^+}^{[p]_q}}(\BF^+, N[1/p]).
	\end{equation}
\end{cor}
\begin{proof}
	The equality in \eqref{eq:h1syn_ext1wach_bf+} follows by inverting $p$ in the complex $\pazs^{\bullet}(N)$.
	Moreover, arguing similar to the proof of Proposition \ref{prop:h1syn_ext1wach_af+} (after inverting $p$), and noting that $(\Fil^k N[1/p]) = (\Fil^k N)[1/p]$ (see Section \ref{subsubsec:nyfil_wachmod_af+}), we see that there exists a natural bijection of sets
	\begin{equation*}
		\alpha \colon H^1(\pazs^{\bullet}(N)[1/p]) \isomorphic \Ext^1_{(\varphi, \Gamma_F)\Mod_{\BF^+}^{[p]_q}}(\BF^+, N[1/p]),
	\end{equation*}
	with an inverse $\beta$ given by a construction similar to the one described in the proof of Proposition \ref{prop:h1syn_ext1wach_af+}.
	It remains to show that $\alpha$ and $\beta$ are additive.

	To this end, let us first explicitly describe the addition structure on the objects involved in the maps $\alpha$ and $\beta$.
	Let $[(x, y)]$ and $[(x', y')]$ denote two classes in $H^1(\pazs^{\bullet}(N)[1/p])$, i.e.\ we have $x, x'$ in $\Fil^{-1} N$ and $y, y'$ in $N$ such that $(1-[p]_q\varphi)x = \nabla_q(y)$ and $(1-[p]_q\varphi)x' = \nabla_q(y')$.
	Then, we see that $[(x, y)] + [(x', y')] = [(x+x', y+y')]$ in $H^1(\pazs^{\bullet}(N)[1/p])$.

	On the other hand, the abelian group structure on $\Ext^1$ is described using the Baer sum (see \cite[\href{https://stacks.math.columbia.edu/tag/010I}{Tag 010I}]{stacks-project}).
	More precisely, if $[E_1]$ and $[E_2]$ denote two extension classes of $\BF^+$ by $N[1/p]$ in $(\varphi, \Gamma_F)\Mod_{\BF^+}^{[p]_q}$, then the class $[E_1] + [E_2]$ is represented by the extension $E$ in the following $(\varphi, \Gamma_F)\equivariant$ commutative diagram of Wach modules over $\BF^+$ with exact rows:
	\begin{equation}\label{eq:baer_sum}
		\begin{tikzcd}
			0 & {N[1/p] \oplus N[1/p]} & {E_1 \oplus E_2} & {\BF^+ \oplus \BF^+} & 0 \\
			0 & N[1/p] & {E_3} & {\BF^+ \oplus \BF^+} & 0 \\
			0 & N[1/p] & E & {\BF^+} & 0,
			\arrow[from=1-1, to=1-2]
			\arrow[from=1-2, to=1-3]
			\arrow["\Sigma"', from=1-2, to=2-2]
			\arrow[from=1-3, to=1-4]
			\arrow[from=1-3, to=2-3]
			\arrow[from=1-4, to=1-5]
			\arrow[equal, from=1-4, to=2-4]
			\arrow[from=2-1, to=2-2]
			\arrow[from=2-2, to=2-3]
			\arrow[from=2-3, to=2-4]
			\arrow[from=2-4, to=2-5]
			\arrow[from=3-1, to=3-2]
			\arrow[equal, from=3-2, to=2-2]
			\arrow[from=3-2, to=3-3]
			\arrow[from=3-3, to=2-3]
			\arrow[from=3-3, to=3-4]
			\arrow["\Delta"', from=3-4, to=2-4]
			\arrow[from=3-4, to=3-5]
		\end{tikzcd}
	\end{equation}
	where $E_3$ is obtained as pushout by the sum map $\Sigma \colon N[1/p] \oplus N[1/p] \rightarrow N[1/p]$, and $E$ is obtained as pullback by the diagonal map $\Delta \colon \BF^+ \rightarrow \BF^+ \oplus \BF^+$.
	More explicitly, arguing similar to the proof of Proposition \ref{prop:h1syn_ext1wach_af+}, for $i = 1, 2$, we may write $E_i = N[1/p] \oplus \BF^+ \cdot e_i$, where $e_i$ in $E_i$ is a lift of the identity element in $\BF^+$, and we have $(\gamma-1)e_i = \mu x_i$ and $(1-\varphi)e_i = y_i$, for some $x_i$ in $\Fil^{-1} N[1/p]$ and $y_i$ in $N[1/p]$.
	Additionally, by the commutativity of the action of $\varphi$ and $\gamma$, we also have that $(1-[p]_q\varphi)x_i = \nabla_q(y_i)$.
	Then, in diagram \eqref{eq:baer_sum}, we note that $(e_1, e_2)$ denotes an element of $E_1 \oplus E_2$ lifting $(1,1)$ in $\BF^+ \oplus \BF^+$.
	Consequently, for the image of $(e_1, e_2)$ in $E_3$, denoted as $e_3$, we have that $(\gamma-1)e_3 = \mu (x_1+x_2)$, $(1-\varphi)e_3 = y_1+y_2$ and $(1-[p]_q\varphi)(x_1+x_2) = \nabla_q(y_1+y_2)$.
	As we have that $\Delta(1) = (1, 1)$ in $\BF^+ \oplus \BF^+$, we let $e$ in $E$ denote the pullback of $e_3$, and note that $e$ is a lifting of the identity element in $\BF^+$.
	In particular, we may write $E = N[1/p] \oplus \BF^+ \cdot e$, and we have $(\gamma-1)e = \mu(x_1+x_2)$, $(1-\varphi)e = y_1+y_2$ and $(1-[p]_q\varphi)(x_1+x_2) = \nabla_q(y_1+y_2)$.
	Arguing similar to the proof of Proposition \ref{prop:h1syn_ext1wach_af+}, it is easy to see that the class $[E]$ is independent of all the choices involved.

	Now, we are ready to show that $\alpha$ and $\beta$ are additive.
	So, let $[(x, y)]$ and $[(x', y')]$ denote two classes in $H^1(\pazs^{\bullet}(N)[1/p])$, and set $[E] := \alpha[(x,y)]$, $[E'] := \alpha[(x',y')]$ and $[E''] = \alpha[(x+x',y+y')]$.
	Then, using the explicit presentation of $[E]$, $[E']$ and $[E'']$ (similar to the one described in the proof of Proposition \ref{prop:h1syn_ext1wach_af+}), the description of Baer sum from the discussion above, and the independence of these classes from the choice of their respective presentations, it follows that $[E] + [E'] = [E'']$.
	In particular, we get that $\alpha$ is a morphism of abelian groups.

	On the other hand, let $[E_1]$ and $[E_2]$ denote two extension classes of $\BF^+$ by $N[1/p]$ in $(\varphi, \Gamma_F)\Mod_{\BF^+}^{[p]_q}$, as in the previous paragraph.
	Then, we have that $\beta([E_i]) = [(x_i, y_i)]$, for $i = 1, 2$, and $\beta([E]) = [(x_1+x_2, y_1+y_2)]$.
	So, it follows that
	\begin{equation*}
		\beta([E_1] + [E_2]) = \beta([E]) = [(x_1+x_2, y_1+y_2)] = [(x_1, y_1)] + [(x_2, y_2)] = \beta([E_1]) + \beta([E_2]),
	\end{equation*}
	i.e.\ $\beta$ is also a morphism of abelian groups.
	This completes our proof.
\end{proof}

\subsubsection{Comparing \texorpdfstring{$H^2$}{-} rationally}

For convenience in computations in this section, let us rephrase our goal.
Let $V$ be a $\padic$ positive crystalline representation of $G_F$, i.e.\ all its Hodge-Tate weights are less than or equal to $0$, and let $T \subset V$ be a $G_F\textrm{-stable}$ $\ZZ_p\lattice$.
Set $V(r) := V \otimes_{\QQ_p} \QQ_p(r)$ and $T(r) := T \otimes_{\ZZ_p} \ZZ_p(r)$, for any $r \in \ZZ$.
From Theorem \ref{thm:crystalline_wach_af+_equivalence}, recall that we have Wach modules $\NF(T)$ and $\NF(T(r)) = \mu^{-r}\NF(T)(r)$ over $\AF^+$, and $\AF^+[1/p]\modules$ $\NF(V) = \NF(T)[1/p]$ and $\NF(V(r)) = \mu^{-r} \NF(V)(r)$.
Let us denote the complex $\pazs^{\bullet}(\mu^{-r}\NF(T)(r))[1/p]$ by $\pazs^{\bullet}(\NF(V), r)$.
Then, our goal is to show the following claim:
\begin{prop}\label{prop:h2_syn_blochkatoselmer}
	The cohomology group $H^2(\pazs^{\bullet}(\NF(V), r))$ vanishes.
	In particular, we have that $H^2(\pazs^{\bullet}(\NF(V), r)) = H^2_f(G_F, V(r)) = 0$.
\end{prop}
\begin{proof}
	Let $x$ be in $\NF(V(r))$, and to prove the claim note that it is enough to show that we can write $x = \nabla_q(y) - (1-[p]_q\varphi)z$, for some $y$ in $\NF(V(r))$ and $z$ in $\Fil^{-1} \NF(V(r))$.
	Write $x = \frac{x'}{\mu^r} \otimes \epsilon^{\otimes r}$, for some $x'$ in $\NF(V)$ and $\epsilon^{\otimes r}$ an $\AF^+\textrm{-basis}$ of $\AF^+(r)$.
	Then, from Lemma \ref{lem:xmuk_h2syn} there exist $y'$ and $z'$ in $\NF(V)$ satisfying the following:
	\begin{equation*}
		\tfrac{x'}{\mu^r} \otimes \epsilon^{\otimes r} = \nabla_q\big(\tfrac{y'}{\mu^{r-1}} \otimes \epsilon^{\otimes r}) - (1-[p]_q\varphi)(z' \otimes \epsilon^{\otimes r}).
	\end{equation*}
	Letting $z = z' \otimes \epsilon^{\otimes r}$ and $y = \tfrac{y'}{\mu^{r-1}} \otimes \epsilon^{\otimes r}$, we see that $x = \nabla_q(y) - (1-[p]_q\varphi)z$, with $y$ in $\NF(V(r))$ and $z$ in $\NF(V)(r) \subset \NF(V(r))$.
	However, note that $[p]_q\varphi(z) = x+z-\nabla_q(y)$ is in $\NF(V(r))$, in particular, $z$ is in $\Fil^{-1} \NF(V(r))$.
	Hence, we get the claim.
\end{proof}

Let $\epsilon^{\otimes r}$ denote an $\AF^+\textrm{-basis}$ of $\AF^+(r)$, and note that the following was used in Proposition \ref{prop:h2_syn_blochkatoselmer}:
\begin{lem}\label{lem:xmuk_h2syn}
	Let $x$ in $\NF(V)$, then for $1 \leqslant k \leqslant r$, there exist some $y$ and $z$ in $\NF(V)$ such that,
	\begin{equation*}
		\tfrac{x}{\mu^k} \otimes \epsilon^{\otimes r} = \nabla_q\big(\tfrac{y}{\mu^{k-1}} \otimes \epsilon^{\otimes r}) - (1-[p]_q\varphi)(z \otimes \epsilon^{\otimes r}).
	\end{equation*}
\end{lem}
\begin{proof}
	Note that $\gamma$ is a topological generator of $\Gamma_0$, in particular, we have that $\chi(\gamma)$ is in $1 + p\ZZ_p$.
	Moreover, note that up to multiplying by some power of $p$ we may assume that $x$ is in $\NF(T)$.
	Therefore, to prove the lemma, it is enough to show that for any $x$ in $\NF(T)$ there exist some $y$ and $z$ in $\NF(V)$ such that
	\begin{equation}\label{eq:xmuk}
		\tfrac{x}{\mu^k} \otimes \epsilon^{\otimes r} = \tfrac{\gamma-1}{\mu}\big(\tfrac{y}{\mu^{k-1}} \otimes \epsilon^{\otimes r}) - (1-[p]_q\varphi)(z \otimes \epsilon^{\otimes r}).
	\end{equation}
	Before proving the modified claim, let us recall that the action of $\Gamma_F$ is trivial on $N_F(T)/\mu \NF(T)$ (see Definition \ref{defi:wachmod_af+}), so for any $x$ in $\NF(T)$ there exists some $x_1$ in $\NF(T)$ such that $(\gamma-1)x = \mu x_1$.
	Using this, we shall prove the claim above by induction on $k$.
	So, let $k = 1$, and note that we have the following:
	\begin{equation*}
		\tfrac{\gamma-1}{\mu}\big(\tfrac{x}{\chi(\gamma)^r-1} \otimes \epsilon^{\otimes r}) = \big(\tfrac{x}{\mu} + \tfrac{\chi(\gamma)^r x_1}{\chi(\gamma)^r-1}\big) \otimes \epsilon^{\otimes r} = \big(\tfrac{x}{\mu} + (1-[p]_q\varphi)z_1) \otimes \epsilon^{\otimes r},
	\end{equation*}
	where $z_1$ is in $\NF(V)$ following Remark \ref{rem:pqphi_converge}.
	Upon rearranging the terms, we see that \eqref{eq:xmuk} holds for $k = 1$.
	Now, we write $u = (\chi(\gamma)\mu)/\gamma(\mu)$ in $1 + p\mu \AF^+$, take $1 < k \leqslant r$, and assume that \eqref{eq:xmuk} holds for $k-1$.
	Then, we have that
	\begin{align*}
		\tfrac{\gamma-1}{\mu}\big(\tfrac{x}{\mu^{k-1}(\chi(\gamma)^{r-k+1}-1)} \otimes \epsilon^{\otimes r}) &= \tfrac{u^{k-1}\chi(\gamma)^{r-k+1}-1}{\mu^k (\chi(\gamma)^{r-k+1}-1)}x \otimes \epsilon^{\otimes r} + \tfrac{u^{k-1}\chi(\gamma)^{r-k+1}}{\mu^{k-1} (\chi(\gamma)^{r-k+1}-1)}x_1 \otimes \epsilon^{\otimes r}\\
		&= \big(\tfrac{x}{\mu^k} + \tfrac{x_k}{\mu^{k-1}}\big) \otimes \epsilon^{\otimes r}\\
		&= \tfrac{x}{\mu^k} \otimes \epsilon^{\otimes r} + \tfrac{\gamma-1}{\mu}\big(\tfrac{y_k}{\mu^{k-2}} \otimes \epsilon^{\otimes r}) - (1-[p]_q\varphi) (z_k \otimes \epsilon^{\otimes r}),
	\end{align*}
	for some $x_k$, $y_k$ and $z_k$ in $\NF(V)$, and note that the last equality follows by the induction hypothesis.
	By rearranging the terms, we see that \eqref{eq:xmuk} also holds for any $1 < k \leqslant r$.
	This allows us to conclude.
\end{proof}

\begin{rem}\label{rem:pqphi_converge}
	For any $x$ in $\NF(T)$, there exists some $y$ in $\NF(T)$ such that $(1-[p]_q\varphi)y = x$.
	Indeed, note that the series $(1 + [p]_q\varphi + ([p]_q\varphi)^2 + \cdots)$ converge as series of operators on $\NF(T)$ since we have that $\prod_{k=0}^n \varphi^k([p]_q)$ is in $(p, \mu)^{n+1}$, for each $n \in \NN$.
	In particular, we see that $([p]_q\varphi)^n$ is $(p, \mu)\textrm{-adically}$ nilpotent, and we may take $y = (1 + [p]_q\varphi + ([p]_q\varphi)^2 + \cdots)x$ in $\NF(T)$.
	It is easy to see that a similar claim is also true for $\NF(V)$.
\end{rem}

\subsection{Syntomic complex over \texorpdfstring{$S$}{-}}\label{subsec:syntomic_complex_s}

Let $M$ be a Wach module over $S$, and define an operator $\nabla_0 := \tfrac{\gamma-1}{\mu_0} \colon M \rightarrow M$.
From Remark \ref{rem:gamma_minus1_image} note that we have $\nabla_0(\Fil^k M) \subset \Fil^{k-p+1} M$, for each $k \in \ZZ$.
\begin{defi}\label{defi:syntomic_complex_s}
	Define the \textit{syntomic complex} with coefficients in $M$ as
	\begin{equation}\label{eq:syntomic_complex_s}
		\pazs^{\bullet}(M) \colon \Fil^0 M \xrightarrow{\hspace{1mm}(\nabla_0, 1-\varphi)\hspace{1mm}} \Fil^{-p+1} M \oplus M \xrightarrow{\hspace{1mm}(1-\ptilde^{p-1}\varphi, \nabla_0)^{\intercal}\hspace{1mm}} M,
	\end{equation}
	where the first map is given as $x \mapsto (\nabla_0(x), (1-\varphi)x)$, and the second map is given as $(x, y) \mapsto (1-\ptilde^{p-1}\varphi)x - \nabla_0(y)$.
\end{defi}

\begin{nota}
	Similar to the notation described before Proposition \ref{prop:h1syn_ext1wach_af+}, we note the following:
	let $M$ and $M'$ be two Wach modules over $S$.
	We say that a Wach module $E$ over $S$ is an extension of $M$ by $M'$ if it fits into a $(\varphi, \Gamma_0)\equivariant$ exact sequence of $S\textrm{-modules}$, $0 \rightarrow M' \rightarrow E \rightarrow M \rightarrow 0$.
	Two extensions $E_1$ and $E_2$ are equivalent if there exists a $(\varphi, \Gamma_0)\equivariant$ commutative diagram of $S\textrm{-modules}$ with exact rows:
	\begin{center}
		\begin{tikzcd}
			0 \arrow[r] & M' \arrow[r] \arrow[d, equal] & E_1 \arrow[r] \arrow[d, "\wr"] & M \arrow[r] \arrow[d, equal] & 0\\
			0 \arrow[r] & M' \arrow[r] & E_2 \arrow[r] & M \arrow[r] & 0.
		\end{tikzcd}
	\end{center}
	Let $\Ext^1_{(\varphi, \Gamma_0)\Mod_{S}^{\ptilde}}(M, M')$ denote the set of equivalence classes of extensions of $M$ by $M'$.
\end{nota}

\begin{prop}\label{prop:h1syn_ext1wach_s}
	Let $M$ be a Wach module over $S$.
	Then, we have a natural bijection of sets
	\begin{equation}\label{eq:h1syn_ext1wach_s}
		H^1(\pazs^{\bullet}(M)) \isomorphic \Ext^1_{(\varphi, \Gamma_0)\Mod_S^{\ptilde}}(S, M).
	\end{equation}
	Moreover, after inverting $p$, we have a natural isomorphism of abelian groups:
	\begin{equation}\label{eq:h1syn_ext1wach_spinv}
		H^1(\pazs^{\bullet}(M))[1/p] = H^1(\pazs^{\bullet}(M)[1/p]) \isomorphic \Ext^1_{(\varphi, \Gamma_0)\Mod_{S[1/p]}^{\ptilde}}(S[1/p], M[1/p]).
	\end{equation}
\end{prop}
\begin{proof}
	For the bijection of sets in \eqref{eq:h1syn_ext1wach_s}, note that the arguments given in the proof of Proposition \ref{prop:h1syn_ext1wach_af+} easily adapts to our current setting.
	In particular, in the proof of Proposition \ref{prop:h1syn_ext1wach_af+} we can replace $\AF^+$ with $S$ and $N$ with $M$, and work in the category $(\varphi, \Gamma_0)\Mod_S^{\ptilde}$ instead of $(\varphi, \Gamma_F)\Mod_{\AF^+}^{[p]_q}$.
	Similarly, adapting the arguments given in the proof of Corollary \ref{cor:h1syn_ext1wach_bf+} to our current setting yields the isomorphism of abelian groups in \eqref{eq:h1syn_ext1wach_spinv}.
	We omit the details to avoid repetition.
\end{proof}

Now, recall that $N := \AF^+ \otimes_S M$ is a Wach module over $\AF^+$ from Theorem \ref{thm:wachmod_s_af+_equiv}.
Our next goal is to compare the syntomic complexes defined in \eqref{eq:syntomic_complex_s} and \eqref{eq:syntomic_complex_af+} with coefficients in $M$ and $N$, respectively.
\begin{prop}\label{prop:syntomic_stoaf+}
	The natural $S\linear$ and $(\varphi, \Gamma_0)\equivariant$ map $M \rightarrow N$ induces a natural morphism of complexes $\pazs^{\bullet}(M) \rightarrow \pazs^{\bullet}(N)$.
	Moreover, the natural map on cohomology $H^k(\pazs^{\bullet}(M)) \rightarrow H^k(\pazs^{\bullet}(N))$ is bijective for $k = 0, 1$ and injective for $k = 2$.
\end{prop}
\begin{proof}
	Let us consider the following diagram:
	\begin{equation}\label{eq:syntomic_stoaf+}
		\begin{tikzcd}[column sep=large]
			C^{\bullet} = \Fil^0 M \arrow[r, "(\gamma-1 {,} 1-\varphi)"] \arrow[d] & \mu_0\Fil^{-p+1} M \oplus M \arrow[r, "(1-\varphi {,} \gamma-1)^{\intercal}"] \arrow[d] &[5mm] \mu_0 M \arrow[d] \\
			D^{\bullet} = \Fil^0 N \arrow[r, "(\gamma-1 {,} 1-\varphi)"] & \mu\Fil^{-1} N \oplus N \arrow[r, "(1-\varphi {,} \gamma-1)^{\intercal}"] & \mu N,
		\end{tikzcd}
	\end{equation}
	where the top row is isomorphic to the complex $\pazs^{\bullet}(M)$ in \eqref{eq:syntomic_complex_s}, the bottom row is isomorphic to the complex $\pazs^{\bullet}(N)$ in \eqref{eq:syntomic_complex_af+}, and the vertical maps are induced by the natural $S\linear$ and $(\varphi, \Gamma_0)\equivariant$ map $M \rightarrow N$.
	In particular, \eqref{eq:syntomic_stoaf+} describes a natural morphism of complexes $\pazs^{\bullet}(M) \rightarrow \pazs^{\bullet}(N)$.

	Now, note that each term of the complex $D^{\bullet}$ in \eqref{eq:syntomic_stoaf+} admits an action of $\Gamma_{\textrm{tor}} \isomorphic \FF_p^{\times}$, which commutes with the action of $\varphi$ and $\Gamma_0$.
	So, by Remark \ref{rem:fpx_decomp}, the complex $D^{\bullet}$ admits a decomposition as $\oplus_{i=0}^{p-1} D_i^{\bullet}$, where
	\begin{equation*}
		D_i^{\bullet} = (\Fil^0 N)_i \xrightarrow{\hspace{1mm}(\gamma-1, 1-\varphi)\hspace{1mm}} (\mu\Fil^{-1} N)_i \oplus N_i \xrightarrow{\hspace{1mm}(1-\varphi, \gamma-1)^{\intercal}\hspace{1mm}} (\mu N)_i.
	\end{equation*}
	Using that $M \isomorphic N_0$ as $(\varphi, \Gamma_0)\modules$ over $S$, and the description of filtration on $M$ in Section \ref{subsubsec:nyfil_m}, it follows that $C^{\bullet} \isomorphic D_0^{\bullet}$.
	Moreover, it is clear that we have $H^k(D^{\bullet}) = \oplus_{i=0}^{p-1} H^k(D_i^{\bullet})$.
	In particular, the natural map $H^k(\pazs^{\bullet}(M)) \isomorphic H^k(C^{\bullet}) \rightarrow H^k(D^{\bullet}) \lisomorphic H^k(\pazs^{\bullet}(N))$ is injective for each $k = 0, 1, 2$.
	Now, for $k =0$, note that since we have $\Fil^0 M \isomorphic (\Fil^0 N)^{\FF_p^{\times}}$ as $(\varphi, \Gamma_0)\modules$ over $S$, therefore, by using Lemma \ref{lem:h0_syn_blochkatoselmer} we get that
	\begin{align*}
		H^0(\pazs^{\bullet}(M)) \isomorphic (\Fil^0 M)^{\varphi=1, \gamma=1} &\isomorphic (\Fil^0 N)^{\varphi=1, \gamma=1, \FF_p^{\times}}\\
		&= N^{\varphi=1, \gamma=1, \FF_p^{\times}} = N^{\varphi=1, \Gamma_F} \lisomorphic H^0(\pazs^{\bullet}(N)).
	\end{align*}
	Finally, for $k = 1$, let us consider the following diagram:
	\begin{center}
		\begin{tikzcd}
			H^1(\pazs^{\bullet}(M)) \arrow[r, "\sim"', "\eqref{eq:h1syn_ext1wach_s}"] \arrow[d] & \Ext^1_{(\varphi, \Gamma_0)\Mod_S^{\ptilde}}(S, M) \arrow[d, "\wr"]\\
			H^1(\pazs^{\bullet}(N)) \arrow[r, "\sim"', "\eqref{eq:h1syn_ext1wach_af+}"] & \Ext^1_{(\varphi, \Gamma_F)\Mod_{\AF^+}^{[p]_q}}(\AF^+, N),
		\end{tikzcd}
	\end{center}
	where the left vertical arrow is the natural map constructed above, and the right vertical arrow is an isomorphism induced by the functor in \eqref{eq:wachmod_s_af+_equiv} of Theorem \ref{thm:wachmod_s_af+_equiv}, which preserves short exact sequences.
	The diagram commutes by definition, and it follows that the left vertical arrow is an isomorphism.
	This allows us to conclude.
\end{proof}

\begin{rem}\label{rem:syntomic_stoaf+}
	From the injectivity of the map $H^2(\pazs^{\bullet}(M)) \rightarrow H^2(\pazs^{\bullet}(N))$, it follows that we have $H^2(\pazs^{\bullet}(M))[1/p] = H^2(\pazs^{\bullet}(N))[1/p] = 0$ (see Proposition \ref{prop:h2_syn_blochkatoselmer}).
\end{rem}

\begin{thm}\label{thm:syntomic_blochkatoselmer_s}
	Let $M$ be a Wach module over $S$ and $V = \TF(M)[1/p]$ the associated $\padic$ crystalline representation of $G_F$ from Theorem \ref{thm:crystalline_wach_s_equivalence}.
	Then, we have a natural isomorphism, for each $k \in \NN$,
	\begin{equation*}
		H^k(\pazs^{\bullet}(M))[1/p] \isomorphic H^k_f(G_F, V).
	\end{equation*}
\end{thm}
\begin{proof}
	The claim follows by combining Proposition \ref{prop:syntomic_stoaf+}, Remark \ref{rem:syntomic_stoaf+} and Theorem \ref{thm:syntomic_blochkatoselmer_af+}.
\end{proof}

\begin{rem}\label{rem:syntomic_to_blochkato_s}
	Similar to Remark \ref{rem:syntomic_to_blochkato_af+}, let us consider the following diagram of complexes:
	\begin{equation}\label{eq:syntomic_to_blochkato_s}
		\begin{tikzcd}[column sep=large]
			\mu_0 \Fil^{-1} M \arrow[r, "(\gamma-1 {,} 1-\varphi)"] \arrow[d] & \mu_0 \Fil^{-p+1} M \oplus \mu_0 M \arrow[r, "(1-\varphi {,} \gamma-1)^{\intercal}"] \arrow[d] &[5mm] \mu_0 M \arrow[d, equal]\\
			\Fil^0 M \arrow[r, "(\gamma-1 {,} 1-\varphi)"] \arrow[d] & \mu_0 \Fil^{-p+1} M \oplus M \arrow[r, "(1-\varphi {,} \gamma-1)^{\intercal}"] \arrow[d] & \mu_0 M \\
			\Fil^0 (M/\mu_0 M) \arrow[r, "1-\varphi"] & M/\mu_0 M,
		\end{tikzcd}
	\end{equation}
	where the vertical arrows are natural maps, and the complex in the middle row is isomorphic to the complex $\pazs^{\bullet}(M)$ in \eqref{eq:syntomic_complex_s} and may be seen as a subcomplex of the Fontaine--Herr complex in \eqref{eq:fontaine_herr_complex}.
	Now, let $V := \TF(M)[1/p]$ from Theorem \ref{thm:crystalline_wach_s_equivalence}.
	Then, after inverting $p$ and using Theorem \ref{thm:mu0_deformation_dcrys}, we see that the complex in the bottom row of \eqref{eq:syntomic_to_blochkato_s} is the same as the complex $\cald^{\bullet}(\Dcrys(V))$ in \eqref{eq:dcrys_complex_bks}.
	Moreover, note that in \eqref{eq:syntomic_to_blochkato_s} the middle column is exact by Theorem \ref{thm:mu0_deformation_dcrys}, and the left-hand side column is exact by \eqref{eq:fil_mmodmu0_ker} in Lemma \ref{lem:fil_mmodmu0_ker}.
	Hence, by Theorem \ref{thm:syntomic_blochkatoselmer_s} and Corollary \ref{cor:dcrys_complex_bks}, it follows that the diagram \eqref{eq:syntomic_to_blochkato_s} induces a natural quasi-isomorphism of complexes $\pazs^{\bullet}(M)[1/p] \simeq \cald^{\bullet}(\Dcrys(V))$.
\end{rem}



\setstretch{0.9}
\phantomsection
\printbibliography[heading=bibintoc, title={References}]

\Addresses

\end{document}